\documentclass[12pt]{amsart}
\usepackage[margin=1in]{geometry}
\usepackage{amsmath,amssymb,amsthm,graphics,mathrsfs, bm}
\usepackage{booktabs, multirow}
\usepackage[all]{xypic}
\usepackage{array}
\usepackage{tikz} \usetikzlibrary{cd}
\usepackage{color}

\usepackage[breaklinks=true, 
colorlinks, linktocpage ]{hyperref}
\usepackage{breakurl}

\numberwithin{equation}{subsection} 
\usepackage[capitalize, nameinlink]{cleveref} 
\crefname{equation}{}{} 
\crefname{subsection}{}{} 
\crefname{section}{\S\kern -.5 ex}{\S\S\kern -1 ex}
\crefname{enumi}{}{} \creflabelformat{enumi}{#2(#1)#3}

\usepackage{cancel}

\theoremstyle{plain}
\newtheorem{corollary}[subsection]{Corollary} 
\newtheorem{lemma}[subsection]{Lemma}
\newtheorem{proposition}[subsection]{Proposition}
\newtheorem{theorem}[subsection]{Theorem} 

\theoremstyle{definition}%

\newcommand\CD{{\mathcal D}}
\newcommand\CH{{\mathscr H}}

\newcommand\CT{{\mathcal T}}

\newcommand\CX{{\mathcal X}}

\newcommand\SA{{\mathscr A}}
\newcommand\SB{{\mathscr B}}
\newcommand\SE{{\mathscr E}}
\newcommand\SF{{\mathscr F}}
\newcommand\SG{{\mathscr G}}

\newcommand\BBC{{\mathbb C}}

\newcommand\BBQ{{\mathbb Q}}

\newcommand\Cent{\operatorname{Cent}}
\newcommand\codim{\operatorname{codim}}

\newcommand\Fix{{\operatorname{Fix}}}
\newcommand\Hom{{\operatorname{Hom}}}
\newcommand\Ind{{\operatorname{Ind}}}
\newcommand\GL{\operatorname{GL}}

\newcommand\inverse{^{-1}}
\renewcommand\th{{^{\text{th}}}}
\newcommand\id{{id}}

\newcommand\Cbar{\overline C}

\newcommand\Gtilde{{\widetilde G}}
\newcommand\htilde{{\widetilde h}}
\newcommand\Htilde{{\widetilde H}}

\newcommand\Hbar{\overline H}
\newcommand\SAbar{\overline {\SA}}

\newcommand\sigmabar{\overline {\sigma}}
\newcommand\Ttilde{{\widetilde T}}

\newcommand\Zbar{\overline Z}

\newcommand\cx{\operatorname{cx}}
\newcommand\Av{\operatorname{Av}}
\newcommand\spn{\operatorname{span}}
\newcommand\cd{\operatorname{cd}}
\newcommand\rk{\operatorname{rk}}
\newcommand\rt{\operatorname{rt}}

\newcommand\mbar{\overline m}
\newcommand\tdi{\operatorname{tdi}}
\newcommand\sigmarbar{\overline {\sigma^r}}

\thanks{The authors would like to thank their charming wives for their
  unwavering support during the preparation of this paper.}



\title [Invariants] {Invariants and semi-invariants in the cohomology of the
  complement of a reflection arrangement}

\author[J.M. Douglass]{J. Matthew Douglass}
\address{Division of Mathematical Sciences,
National Science Foundation,
2415 Eisenhower Ave,
Alexandria, VA 22314, USA}
\email{mdouglas@nsf.gov}

\author[G. Pfeiffer]{G\"otz Pfeiffer} \address{School of Mathematical and Statistical Sciences, University of Galway,
  Galway, Ireland}
\email{goetz.pfeiffer@universityofgalway.ie}

\author[G. R\"ohrle]{Gerhard R\"ohrle} \address {Fakult\"at f\"ur
  Mathematik, Ruhr-Universit\"at Bochum, D-44780 Bochum, Germany}
\email{gerhard.roehrle@rub.de}

\begin{document}

\begin{abstract}
  Suppose $V$ is a finite dimensional, complex vector space, $\SA$ is a finite
  set of codimension one subspaces of $V$, and $G$ is a finite subgroup of the
  general linear group $\GL(V)$ that permutes the hyperplanes in $\SA$. In
  this paper we study invariants and semi-invariants in the graded $\BBQ
  G$-module $H^*(M(\SA))$, where $M(\SA)$ denotes the complement in $V$ of the
  hyperplanes in $\SA$ and $H^*(\,\cdot\,)$ denotes rational singular
  cohomology, in the case when $\SA$ is a reflection arrangement and the pair
  $(\SA,G)$ arises from a reflection coset. The main result is the
  construction of an explicit, natural (from the point of view of Coxeter
  groups) basis of the space of invariants, $H^*(M(\SA))^G$. In addition to
  leading to a proof of the description of the space of invariants conjectured
  by Felder and Veselov for Coxeter groups that does not rely on computer
  calculations, this construction provides an extension of this description of
  the space of invariants to arbitrary finite, complex reflection groups. The
  main result also leads to simplifications of some cohomology computations of
  Lehrer, Callegaro-Marin, and Marin.
\end{abstract}

\maketitle

\tableofcontents

\allowdisplaybreaks


\section{Introduction}

\subsection{}\label{ssec:1.1}

Let $V$ be a finite dimensional, complex vector space. An
\emph{arrangement-group pair with underlying vector space $V$} is a pair
$(\SA, G)$, where $\SA$ is a central hyperplane arrangement in $V$, or more
simply an ``arrangement,'' that is, $\SA$ is a finite set of codimension one
subspaces of $V$, and $G$ is a finite subgroup of the general linear group
$\GL(V)$ that permutes the hyperplanes in $\SA$. The ``complement of $\SA$''
is the open submanifold
\[
  M(\SA) =V\setminus \bigcup_{H\in \SA} H
\]
of $V$. Clearly, $M(\SA)$ is a $G$-stable submanifold of $V$. Let
\[
  H^*(M(\SA)) = \bigoplus_{k\geq 0} H^k(M(\SA))
\]
denote the rational, singular cohomology of $M(\SA)$.

The rule $g\mapsto (g\inverse)^*$ endows $H^*(M(\SA))$ with the structure of a
graded $\BBQ G$-algebra and the space of invariants, $H^*(M(\SA))^G$, can be
compactly encoded in the \emph{Poincar\'e polynomial}
\[
  P(\SA,G; t)= \sum_{k\geq 0} \dim \big(H^k(M(\SA))^G \big)\, t^k.
\]

In this paper we study invariants and semi-invariants in graded $\BBQ
G$-modules $H^*(M(\SA))$ that arise in the context of complex reflection
groups, and more generally, reflection cosets.

The main results are (1) the construction of an explicit, natural (from the
point of view of Coxeter groups) basis of $H^*(M(\SA))^G$, where $(\SA, G)$
arises as described below from a reflection coset, and (2) a short new proof
that determinant-like characters do not occur in $H^*(M(\SA(G)))$ for a
complex reflection group $G$.

\subsection{}\label{ssec:1.1a}

In more detail, recall that a non-identity linear transformation with finite
order in $\GL(V)$ is a \emph{reflection} if it fixes a hyperplane in $V$
pointwise, a finite subgroup of $\GL(V)$ is a \emph{(complex or unitary)
  reflection group} if it is generated by reflections, and a \emph{reflection
  coset} is a finite subset, $C\subseteq \GL(V)$, such that (1) $CC\inverse =
\{\, gh\inverse\mid g, h\in C\,\}$ is a reflection subgroup of $\GL(V)$ that
is normal in $\langle C \rangle$, the subgroup generated by $C$, and (2) the
quotient $\langle C \rangle/CC\inverse$ is a finite cyclic group. In
particular, $\langle C\rangle$ is a finite subgroup of $\GL(V)$, which may or
may not be a reflection group.

Technically arrangements, reflection groups, and reflection cosets are pairs,
namely $(\SA,V)$, $(G,V)$, and $(C, V)$, respectively. In practice the ambient
vector space is omitted when it is clear from context. Notice that the
identity linear transformation is not a reflection but that the trivial
subgroup of $\GL(V)$ is a reflection (sub)group.

For a finite subset $S\subseteq \GL(V)$, define $\SA(S)$ to be the (possibly
empty) arrangement of hyperplanes in $V$ that are fixed point sets of the 
reflections in $S$:
\[
  \SA(S)=\{\, \Fix(r)\mid \text{$r\in S$ is a reflection}\,\},
\]
where $\Fix(r)=\Fix_V(r)$ is the set of fixed points of $r$ in $V$. If $S=G$
is a group, then $\SA(G)$ is a \emph{reflection arrangement.} Obviously
$\SA(G)$ is determined by the subgroup generated by all reflections in $G$.

\subsection{Reflection pairs}\label{ssec:1rp}

Let $C\subseteq \GL(V)$ be a reflection coset. The groups $G_0=CC\inverse$ and
$G_1=\langle C\rangle$ both act on the reflection arrangements
$\SA_0=\SA(G_0)$ and $\SA_1=\SA(G_1)$. In this way, $C$ gives rise to four
arrangement-group pairs as in \cref{ssec:1.1}, namely
\[
  (\SA(G_0), G_0),\quad (\SA(G_0), G_1),\quad (\SA(G_1),
  G_0),\quad\text{and}\quad (\SA(G_1), G_1).
\]

Notice that if $z\in \GL(V)$ is a scalar transformation with finite order,
then $Cz$ is also a reflection coset. Set $G_0'= (Cz)(Cz)\inverse$ and
$G_1'=\langle Cz \rangle$. Then clearly $G_0'= G_0$ and $G_1' = G_1 \langle
z\rangle$. If $\SA$ is any arrangement in $V$, then $z$ acts trivially on
$\SA$ and it is easy to check (see \cref{ssec:osp2}) that $z$ acts trivially
on $H^*(M(\SA))$ as well. Thus, with the obvious notational convention,
\begin{align*}
  H^*(M(\SA_0'))^{G_0'} &= H^*(M(\SA_0))^{G_0},%
  & H^*(M(\SA_0'))^{G_1'} &= H^*(M(\SA_0))^{G_1},\\
  H^*(M(\SA_1'))^{G_0'} &= H^*(M(\SA_1'))^{G_0},%
  & H^*(M(\SA_1'))^{G_1'} &= H^*(M(\SA_1'))^{G_1}.
\end{align*}
In particular, replacing $C$ by $Cz$ only has the affect of introducing one
new arrangement, namely $\SA(\langle Cz\rangle)$, but the groups of linear
transformations of cohomology spaces do not change.

For the purposes of this paper we define a \emph{reflection pair} to be an
arrangement-group pair, $(\SA,G)$, with the property that there is a
reflection coset, say $C$, such that with the notation above,
\[
  H^*(M(\SA))^G= H^*(M(\SA_i))^{G_j}
\]
for some $i,j\in \{0,1\}$. In this case, we say that the reflection pair
$(\SA,G)$ \emph{arises from $C$.} Obviously a given reflection pair can arise
from many reflection cosets.

Notice that if $(\SA, G)$ is a reflection pair with underlying vector space
$V$, then there is a reflection subgroup, $\Gtilde\subseteq \GL(V)$, such that
$G$ normalizes $\Gtilde$ and $\SA=\SA(\Gtilde)$. In particular, although $G$
might not be a reflection group, $\SA$ is a reflection arrangement.

\subsection{}

A reflection coset, $C\subseteq \GL(V)$, is called \emph{irreducible} if
$CC\inverse$ acts irreducibly on $V$. Otherwise $C$ is \emph{reducible.}
Obviously if $C$ is irreducible, then $\langle C\rangle$ also acts irreducibly
on $V$. A reflection pair $(\SA, G)$ is called \emph{irreducible} if it arises
from an irreducible reflection coset.

Irreducible reflection cosets have been classified, up to multiplication by a
scalar transformation, by Brou\'e, Malle, and Michel \cite
{brouemallemichel:spetsesI}. This classification is recalled in
\cref{ssec:rc}. As indicated above, the classification of irreducible
reflection pairs is somewhat more involved. This classification is given in
\cref{cor:irp} and \cref{ssec:firp}.

\subsection{}

With these preliminaries in hand, we can more precisely formulate the main
results in the paper. First, the Poincar\'e polynomials for all reflection
pairs can be determined from \cref{thm:3} (see \cref{cor:1}). Next, explicit
bases of $H^k(M(\SA))^G$, when $(\SA, G)$ is an irreducible reflection pair,
are constructed in \cref{thm:4}. Finally, an immediate consequence of
\cref{thm:5} is that certain linear characters do not occur in
$H^*(M(\SA(G)))$ for any complex reflection group $G$.

\subsection{Poincar\'e polynomials and prior work}

The computation of the Poincar\'e polynomials, $P(\SA(G),G;t)$, for most of
the irreducible reflection pairs can be extracted from work of Lehrer
\cite{lehrer:rationalpoints}, Callegaro, and Marin
\cite{callegaromarin:homology}, \cite{marin:homology}. When $\SA$ is the
arrangement of an imprimitive complex reflection group the argument in
\cite{lehrer:rationalpoints} relies on reduction to positive characteristic
and the Grothendieck trace formula. When $\SA$ is the arrangement of a
primitive complex reflection group one needs to use the deep result that the
orbifold $M(\SA) /G$ is a $K(\pi,1)$-space. The approach used here, which is
an extension of the computations for Coxeter groups in
\cite{brieskorn:tresses}, is more elementary, less computationally intensive
(for primitive groups), and yields a more refined description of the
invariants $H^*(M(\SA(G)))^G$. The arguments do not rely on any machine
computations, although case-by-case calculations are still necessary.

The main results in this paper (see \cref{cor:1}, \cref{cor:2}, and
\cref{thm:4}) suggest the existence of some, as yet unidentified, geometric or
algebraic phenomenon that would provide a conceptual explanation of the
structure of the cohomology ring $H^*(M(\SA(G)))^G$.

\subsection{The Felder-Veselov construction}

The explicit bases of the spaces of invariants constructed in \cref{sec:basis}
also lead to a simple proof of a conjecture for Coxeter groups made by Felder
and Veselov \cite{felderveselov:coxeter}, as well as a natural extension of
the resulting theorem to all reflection pairs (see \cref{thm:4}).

In more detail, consider the special case when $G\subseteq \GL(V)$ is a
reflection group that contains a Coxeter system. The (ungraded) character of
$G$ on $H^*(M(\SA(G)))$ has been computed by various authors. Felder and
Veselov characterize the conjugacy classes in the support of this character as
the conjugacy classes of so-called special involutions. Special involutions
can be classified for each Coxeter type and it turns out that the number of
such classes is equal to the dimension of the invariant subspace
$H^*(M(\SA))^G$. Suppose $g\in G$ and let $Z=Z_G\big(\Fix(g) \big)$ be the
pointwise stabilizer of the space $\Fix(g)$. By a theorem of Steinberg, $Z$ is
a reflection subgroup of $\GL(V)$, and hence a Coxeter group. Let $c$ be a
Coxeter element in $Z$. Then $c$ determines a cohomology class $\zeta_c\in
H^*(M(\SA(G)))$ in a natural way (see \cite{felderveselov:coxeter} or
\cref{sec:basis}). Obviously the average of $\zeta_c$ over $G$, say $\Av_G
(\zeta_c)$, lies in $H^*(M(\SA(G)))^G$, and it is easy to see that $\Av_G
(\zeta_c)$ does not depend on the choice of the Coxeter element $c$. Felder
and Veselov conjectured that the rule $g\mapsto \Av_G(\zeta_c)$ defines a
bijection between the set of conjugacy classes of special involutions and a
basis of $H^*(M(\SA(G)))^G$.

Felder and Veselov proved their conjecture in many cases. The remaining cases
were settled in \cite{douglasspfeifferroehrle:invariants}. Computer
calculations were used to verify the conjecture for the exceptional Coxeter
groups. The approach used in this paper to compute $H^*(M(\SA))^G$, when
$(\SA, G)$ is a reflection pair, leads to a proof of the Felder-Veselov
conjecture that does not rely on computer calculations. The key observation is
that when $G$ is a Coxeter group, the parabolic subgroups that arise in
\cref{thm:4} are precisely the subgroups that arise as pointwise stabilizers
of fixed points of special involutions. The proof of \cref{thm:4} is a
case-by-case computation using the classification of complex reflection
groups.

\subsection{}

Brou\'e, Malle, and Rouquier \cite{brouemallerouquier:complex} define specific
generating sets for non-Coxeter complex reflection groups that have some of
the properties of Coxeter generating sets of a Coxeter group. In this paper,
these generating sets are used to construct bases of $H^*(M(\SA(G)))^G$ that
for Coxeter groups agree with the Felder-Veselov bases. For well-generated
complex reflection groups that are not Coxeter groups, the construction of the
basis of $H^*(M(\SA(G)))^G$ is the same as for Coxeter groups. For groups that
are not well-generated, the situation is more complicated, but still
manageable.

Computations for non-Coxeter complex reflection groups show that the number of
conjugacy classes in the support of the character of $G$ on $H^*(M(\SA(G)))$
can be much larger than the dimension of $H^*(M(\SA(G)))^G$. It would be
interesting to find a concise formula for the character of $G$ on
$H^*(M(\SA(G)))$ in the spirit of \cite {felderveselov:coxeter}, as well as to
identify a suitable replacement of the notion of special involutions that
would be valid for all complex reflection groups.

\subsection{}
The rest of this paper is organized as follows. In the next section we recall
the general results about hyperplane arrangements and reflection arrangements
needed in this paper. In \cref{sec:red} we describe the reduction of the
computation of $H^*(M(\SA))^G$, when $(\SA,G)$ is a reflection pair, to the
case of irreducible reflection pairs. The classification of irreducible
reflection pairs is given in \cref{sec:cla}. \cref{thm:3}, which gives a
decomposition of the spaces of invariants $H^*(M(\SA))^G$ for irreducible
reflection pairs with rank greater than two, is stated in \cref{sec:thm3} and
proved in \cref{sec:prf}. In \cref{sec:basis} explicit bases of the invariants
$H^*(M(\SA))^G$ are described for each irreducible reflection pair with rank
greater than two. In \cref{sec:leh} we show how the bases constructed in
\cref{sec:basis} lead to a refinement of a theorem proved by Lehrer \cite
{lehrer:rationalpoints}. In \cref{sec:semi} we give a short alternative proof
of another result of Lehrer \cite {lehrer:vanishing} that irreducible
characters which do not vanish on subgroups generated by generating
reflections of $G$ do not occur in $H^*(M(\SA(G)))$ when $G$ is a complex
reflection group. The appendix contains computations for rank two reflection
pairs.

\subsection{Conventions and notation}

Basic references are Bourbaki \cite{bourbaki:groupes} and
Lehrer-Taylor \cite{lehrertaylor:unitary} for the theory of complex reflection
groups and the book of Orlik and Terao \cite{orlikterao:arrangements} for the
theory of arrangements.

For the imprimitive reflection groups we need notation for roots of
unity. Define
\[
  \text{$\mu_r$ to be the group of $r\th$ roots of unity}\quad \text{and}\quad
  \omega_r = e^{2\pi \sqrt{-1}/r},
\]
when $r$ is a positive integer.

If $G\subseteq \GL(V)$ is a complex reflection group, we may choose a
$G$-invariant unitary form on $V$. Normally we assume that such a form has
been chosen. Thus, if $X$ is a subspace of $V$, then $X^{\perp}$ denotes the
orthogonal complement of $X$.

Except in \cref{sec:semi}, where $\BBQ$ is replaced by $\BBC$, if $X$ is a
topological space, then $H^*(X)$ denotes the (total) singular cohomology of
$X$ with coefficients in $\BBQ$. If a group $G$ acts on $X$ by homeomorphisms,
we frequently denote the induced action on cohomology, $h\mapsto
(g\inverse)^*(h)$ for $g\in G$ and $h\in H^*(X)$, by $g\cdot h$, or simply
$gh$. The meaning should always be clear by context.

Suppose $\SA$ is an arrangement in a finite dimensional complex vector space
$V$. For a subspace $X\subseteq V$ define
\[
  \cd X= \codim _VX=\dim X^{\perp} \quad\text{and}\quad \SA_X =\{\, H\in
  \SA\mid X\subseteq H \,\},
\]
so $\SA_X$ is an arrangement in $V$. The \emph{center of $\SA$} is $\Cent(\SA)=
\bigcap_{H\in \SA} H$ and $\SA$ is \emph{essential} if $\Cent(\SA)=0$. The
rank of $\SA$, denoted by $\rk \SA$, is defined by
\[
  \rk \SA= \cd\Cent(\SA)=\dim \Cent(\SA)^{\perp}.
\]

The lattice of $\SA$, denoted by $L(\SA)$, is the set of subspaces of $V$ of
the form $H_1\cap \dotsm \cap H_k$, where $\{H_1, \dots, H_k\}$ is a subset of
$\SA$. Notice that if $X\in L(\SA)$, then $\rk \SA_X= \cd X$ and
$\Cent(\SA_X)=X$. For a non-negative integer $k$ set
\[
  L(\SA)_k = \{\, X\in L(\SA) \mid \cd X=k\,\}.
\]

For a subgroup $G\subseteq \GL(V)$ and a subset $X\subset V$ define
\begin{align*}
  Z_G(X)%
  &=\{\, g\in G\mid X\subseteq \Fix_V(g)\,\}, \text{\ the pointwise
    stabilizer of $X$ in $G$, and} \\
  N_G(X)%
  &=\{\, g\in G \mid g X =X\,\}, \text{\ the setwise stabilizer of $X$ in $G$.}
\end{align*}
If in addition $G$ is a reflection subgroup, then the \emph{rank of $G$} is
defined to be the rank of $\SA(G)$, so
\[
  \rk G= \cd \Cent(\SA(G)).
\]

\subsection{Reduction to essential arrangements}\label{ssec:ess}
Suppose $(\SA,G)$ is an arrangement-group pair with underlying vector space
$V$. Let $V_1$ be a subspace of $V$ contained in $\Cent(\SA)$, set
$V_0=V/V_1$, and let $p\colon V\to V_0$ be the projection. Then
\begin{itemize}
\item $p(\SA)=\{\,p(H) \mid H\in \SA\}$ is an arrangement in $V_0$ with
  $L(p(\SA)) = p(L(\SA))$, which is essential if
  $V_1=\Cent(\SA)$,
\item $p$ induces a bijection $\SA\to p(\SA)$ and a lattice isomorphism
  $L(\SA)\to L(p(\SA))$,
\item $M(p(\SA)) = p(M(\SA))$, and
\item $p^*\colon H^*(p(M(\SA))) \xrightarrow{\ \cong\ } H^*(M(\SA))$ is an
  isomorphism.
\end{itemize}

Clearly, $\Cent(\SA)$ is a $G$-stable subspace. Suppose that $V_1$ is also
$G$-stable. Abusing notation slightly, let $p(G)$ denote the image of $G$ in
$\GL(V_0)$. Then
\begin{itemize}
\item $(p(\SA), p(G))$ is an essential arrangement-group pair with underlying
  vector space $V_0$,
\item the isomorphism $p^*$ intertwines the action of $p(G)$ on
  $H^*(M(p(\SA)))$ with the action of $G$ on $H^*(M(\SA))$, and
\item $p^*$ induces an isomorphism of vector spaces $H^*(M(p(\SA))) ^{p(G)}
  \cong H^*(M(\SA))^G$.
\end{itemize}

\section{Preliminaries} \label{sec:setup}

In this section we fix notation and review some background results for future
reference. Unless otherwise indicated, $(\SA,G)$ is an arrangement-group pair
with underlying vector space $V$.

\subsection{The Orlik-Solomon presentation of
  $\bm{H^*(M(\SA))$}} \label{ssec:osp1}

The presentation of $H^*(M(\SA))$, given by Orlik and Solomon
\cite{orliksolomon:combinatorics} for any arrangement $\SA$, underlies the
computations below. For $H\in \SA$ choose a linear form $\alpha_H$ in the dual
space, $V^*$, of $V$ such that $H=\ker \alpha_H$, and define $h$ to be the
cohomology class of the $1$-form $(1/2\pi i) (d\alpha_H/\alpha_H) =
\alpha_H^*((1/2\pi i) (dz/z))$ in the de Rahm cohomology group
$H^1_{dr}(M(\SA))$.

We use the convention that hyperplanes are denoted using upper case $H$,
possibly with decorations, and the corresponding Orlik-Solomon generators of
$H^*(M(\SA))$ are denoted using lower case $h$, with the same decorations. For
example, for any two hyperplanes $H_1$ and $H_2$ in $\SA$, $h_1h_2$ denotes
the product of the Orlik-Solomon generators $(1/2\pi i)
(d\alpha_{H_1}/\alpha_{H_1})$ and $(1/2\pi i) (d\alpha_{H_2}/\alpha_{H_2})$ in
$H^1_{dr}(M(\SA))$. With this convention, the $\BBQ$-subalgebra of $H^*_{dr}
(M(\SA))$ generated by $\{1\}$ and $\{\, h\mid H\in \SA\,\}$ is isomorphic as
a graded algebra to $H^*(M(\SA))$, and
\begin{equation}
  \label{eq:osrel}
  \sum_{i=1}^m (-1)^i h_1 \dotsm \widehat{h_i} \dotsm h_m =0
  \qquad\text{$(\widehat{h_i}$ is deleted)}
\end{equation}
whenever $\{ \alpha_{H_1}, \dots, \alpha_{H_m} \}$ is a linearly dependent
subset of $V^*$. Condition \cref{eq:osrel}, together with the
anti-commutativity condition $h_1h_2 = -h_2h_1$ for $H_1, H_2\in \SA$, is the
Orlik-Solomon presentation of $H^*(M(\SA))$.

\subsection{} \label{ssec:brid}
It follows easily from \cref{ssec:osp1}\cref{eq:osrel} that $H^{k}(M(\SA))$ is
spanned by the $k$-fold products $h_1 \dotsm h_k$ where $\{ \alpha_{H_1},
\dots, \alpha_{H_k} \}$ is a linearly independent subset of $V^*$. In
particular, if $\cd X =k$, then using Brieskorn's Lemma below we identify $H^{\cd X}
( M(\SA_{X}))$ with the subspace of $H^k(M(\SA))$ spanned by all products $h_1
\dotsm h_k$, where $H_1, \dots, H_k\in \SA$ and $H_1\cap \dotsm \cap H_k =X$.

\subsection{} \label{ssec:osp2}

If $(\SA,G)$ is an arrangement-group pair, then $G$ acts as graded algebra
automorphisms on $H^*(M(\SA))$. Thus, if $g\in G$ and $H_1, \dots, H_k\in
\SA$, then
\begin{equation}
  \label{eq:7}
  g\cdot h_1 \dotsm h_k= (g\cdot h_1) \dotsm (g\cdot h_k).
\end{equation}
If $r\in G$ is a reflection and $H_r=\Fix(r) \in \SA$, then
\begin{equation}
  \label{eq:2}
  gH_r= H_{grg\inverse}\in \SA\quad \text{and}\quad g \cdot h_r =
  h_{grg\inverse} \in H^1(M(\SA)).
\end{equation}
It follows easily that $H^1(M(\SA))$ affords the permutation representation
arising from the action of $G$ on $\SA$.

The equality \cref{eq:2} is used without comment in the computations below.

\subsection{Brieskorn's Lemma, group actions, and
  semi-invariants}\label{ssec:br}

The main general result regarding the $\BBQ G$-module structure of the
cohomology ring $H^*(M(\SA))$ used in this paper is an equivariant version of
Brieskorn's Lemma that describes each subspace $H^k(M(\SA))$ as a sum of
induced modules due to Lehrer and Solomon.

Let $\CX(\SA,G)$ be a fixed set of orbit representatives of the action of $G$ on $L(\SA)$ and set
$\CX(\SA,G)_k= \CX(\SA,G) \cap L(\SA)_k$.

Let $\xi\colon G\to \BBQ^{\times}=\GL(\BBQ^1)$ be a rational, linear character
of $G$ and let
\[
  e_\xi = \frac 1 {|G|} \sum_{g\in G}\xi(g\inverse) g
\]
be the centrally primitive idempotent in $\BBQ G$ corresponding to $\xi$. For
a $\BBQ G$-module $M$, set
\[
  M^\xi = e_\xi M = \{\, m\in M \mid \forall g\in G,\,g m = \xi(g) m \,\}.
\]
In the special case when $\xi=1_G$ is the trivial representation,
$M^{1_G}=M^G$ as usual. To minimize subscripts, set $e_G=e_{1_G}$. If $H$ is a
subgroup of $G$, then $\xi|_H$ is a linear character of $H$ with centrally
primitive idempotent $e_{\xi|_H}$ in $\BBQ H$. Note that $e_\xi e_{\xi|_H} =
e_\xi$.

In the next proposition, for $Y\in L(\SA)$ we identify $H^{k} ( M(\SA_{Y}))$
with a subspace of $H^k(M(\SA))$ as in \cref{ssec:brid}.

\begin{proposition}\label{pro:bries}
  Suppose $k\geq0$ and set $\CX_k= \CX(\SA,G)_k$.
  \begin{enumerate}
  \item The inclusions $M(\SA) \subseteq M(\SA_X)$ induce isomorphisms of
    $\BBQ G$-modules, \label{it:bries1}
    \begin{equation*} \label{eq:HK} H^k(M(\SA))\cong \bigoplus_{X\in L(\SA)_k}
      H^k(M(\SA_X)) \cong \bigoplus_{X\in \CX_k} \Ind_{N_G(X)}^G \big(H^{k}(
      M(\SA_{X})) \big).
    \end{equation*}
  \item If $X\in L(\SA)$, then
    multiplication by $e_\xi$ defines an isomorphism \label{it:bries2}
    \[
      e_{\xi_X} H^{k}(M(\SA_{X} ))= H^{k}(M(\SA_{X} ))^{\xi_X} \xrightarrow[\
      \cong\ ]{\ e_\xi\cdot(\,)\ } \Big( \bigoplus_{Y\in G X} H^{k}(M(\SA_{Y}
      )) \Big)^\xi , 
    \]
    where $GX$ denotes the $G$-orbit of $X$. Summing over $X\in \CX_k$, the
    first isomorphism in \cref{eq:HK} restricts to the equality
    \[
      H^k(M(\SA))^\xi =\sum_{X\in \CX_k} e_\xi \cdot H^{k} ( M(\SA_{X}))
      ^{\xi_X} ,
    \]
    where the sum on the right-hand side is an internal direct sum.
  \item Suppose that $\SA= \SA(\Gtilde)$ is a reflection arrangement, where
    $\Gtilde\subseteq \GL(V)$ is a reflection group that is normalized by
    $G$. If $X\in L(\SA)$, then $\SA_X= \SA(Z_{\Gtilde}(X))$, and
    so \label{it:bries3}
    \begin{equation*} \label{eq:tk1}
      H^k(M(\SA))^G \cong \bigoplus_{X\in \CX(\SA, G)_k} H^{k} (
      M(\SA(Z_{\Gtilde}(X)))) ^{N_G(X)}. 
    \end{equation*}
  \end{enumerate}
\end{proposition}

\begin{proof}
  The first isomorphism in \cref{it:bries1} is due to Brieskorn
  \cite{brieskorn:tresses}. The second is stated and proved by Lehrer and
  Solomon \cite{lehrersolomon:symmetric}.

  It is straightforward to check that the projection mapping from
  $\bigoplus_{Y\in G X} H^{k}(M(\SA_{Y} ))$ to $H^{k}(M(\SA_{X}))$ restricts
  to give a left inverse to the multiplication mapping from $H^{k}(M(\SA_{X}
  ))^{\xi_X}$ to $e_\xi\cdot \big(\bigoplus_{Y\in G X} H^{k}(M(\SA_{Y}
  ))\big)$. The assertions in \cref{it:bries2} then follow from
  \cref{it:bries1} and Frobenius reciprocity.

  With the hypotheses of \cref{it:bries3}, if $X\in L(\SA)$, it follows from a
  theorem of Steinberg \cite{steinberg:differential} that $Z_{\Gtilde}(X)$ is
  generated by the reflections it contains, which are precisely the
  reflections that fix the hyperplanes in $\SA_X$ pointwise. Thus,
  $Z_{\Gtilde}(X)$ is a reflection group with $\SA(Z_{\Gtilde}(X))=\SA_X$ and
  $\Cent( \SA(Z_{\Gtilde}(X))) =X$. In particular, $\rk Z_{\Gtilde}(X)= \cd
  X$.
\end{proof}

Notice that the proposition holds when the base field, $\BBQ$, is replace by
$\BBC$ throughout.

\subsection{}\label{ssec:tdi}

Some special cases for small and large $k$ are important in the sequel.
\begin{itemize}
\item In degree $0$, $L(\SA)_0= \{V\}= \CX(\SA,G)_0$. It is easy to see that
  $G$ acts trivially on $H^0(V) \cong H^0(M(\SA))$. Thus $H^0(M(\SA))^G$ is
  one-dimensional and $H^0(M(\SA))^\xi=0$ for $\xi\ne 1_G$.

\item In degree $1$, $L(\SA)_1=\SA$ and $\CX(\SA,G)_1$ is a set of orbit
  representatives for the action of $G$ on $\SA$. It was observed in
  \cref{ssec:osp2} that $H^1(M(\SA))$ is simply the permutation module arising
  from the action of $G$ on $\SA$ and that if $H\in \CX(\SA,G)_1$, then $H^{1}
  ( M(\SA_{H})) ^{N_G(H)}$ is one-dimensional with basis the $N_G(H)$-average
  of the Orlik-Solomon generator $h$.

\item At the other extreme, if $k>\rk \SA$, then $H^k(M(\SA))=0$, and thus
  $H^k(M(\SA))^G=0$.

\item The ``top'' degree is $\rk \SA$, in which case $L(\SA)_{\rk \SA} =
  \{\Cent(\SA)\} =\CX_{\rk \SA}$. Clearly, $\SA_{\Cent(\SA)}= \SA$ and
  $N_G(\Cent(\SA))=G$, so Brieskorn's Lemma does not provide any insight into
  the $\BBQ G$-module structure of $H^{\rk \SA}(M(\SA))$, or $\dim H^{\rk
    \SA}(M(\SA))^G$. However, together with \cref{ssec:cpp}\cref{eq:cn}, the
  isomorphism in \cref{pro:bries}\cref{it:bries3} provides a powerful
  inductive, or recursive, tool to compute $H^{\rk \SA}(M(\SA))^G$ when
  $(\SA,G)$ is an irreducible reflection pair.
\end{itemize}

We say that $(\SA,G)$ \emph{has top degree invariants} if $H^{\rk
  \SA}(M(\SA))^G \ne 0$. Roughly speaking, when $(\SA,G)$ is an irreducible
reflection pair, \cref{pro:bries}\cref{it:bries3} reduces the computation of
$H^k(M(\SA))^G$ to those $X$ in $\CX(\SA,G)_k$ such that the pair $(\SA_X,
Z_{\Gtilde}(X))$ has top degree invariants.

\subsection{Contraction with the Euler vector field}\label{ssec:cont}

The last general construction we need is the acyclic complex from
\cite[\S3.1]{orlikterao:arrangements} that arises from contracting (classes
of) differential forms in $H^*(M(\SA))$ with the Euler vector
field. Precisely, there is a graded derivation of $H^*(M(\SA))$ with degree
$-1$ that maps each Orlik-Solomon generator $h\in H^1(M(\SA))$ to $1\in
H^0(M(\SA))$: 
\[
  \partial\colon H^*(M(\SA))\to H^*(M(\SA)) \quad\text{by} \quad
  \partial(h_1\dotsm h_k) = \sum_{i=1}^k (-1)^{i-1} h_1\dotsm \widehat {h_i}
  \dotsm h_k,
\]
for hyperplanes $H_1, \dots, H_k\in \SA$. It is shown in \cite[Lem.~3.13]
{orlikterao:arrangements} that $(H^*(M(\SA)), \partial)$ is an acyclic chain
complex and it follows from \cref{ssec:osp2}\cref{eq:7} that $\partial$ is
$G$-equivariant. In particular,
\begin{equation*}
  \label{eq:4}
  \partial\colon H^{\rk \SA}(M(\SA)) \to H^{\rk \SA -1}(M(\SA))
\end{equation*}
is injective and $G$-equivariant.

Define
\[
  \mu\colon H^*(M(\SA))\to H^*(M(\SA)) \quad\text{by } \quad \mu(x)=
  \big(|\SA|\inverse \sum_{H\in \SA} h\big) x,
\]
so $\mu$ is left multiplication by the class $|\SA|\inverse \sum_{H\in \SA}
h\in H^1(M(\SA))$. It is straightforward to check that $\mu$ is a graded,
$G$-equivariant, linear transformation with degree $1$ and that $\mu \partial
+\partial \mu = \id$.

For $0\leq k \leq \rk(\SA)$ set
\[
  \Hbar^k= \partial\big( H^{k+1}(M(\SA)) \big).
\]
Then the restriction of $\mu$ to $\Hbar^{k}$ is a $\BBQ G$-module isomorphism
onto its image, with inverse given by the restriction of $\partial$ to $\mu
\big( \Hbar^{k} \big)$. Moreover there is a canonical direct sum decomposition
of $\BBQ G$-modules,
\begin{equation}
  \label{eq:pmu}
  H^k(M(\SA)) \cong \mu\big( \Hbar^{k-1}\big) \oplus \Hbar^k,
\end{equation}
where by convention we take $H^{-1}(M(\SA))=0$.

\subsection{} \label{ssec:cpp}

Two immediate consequences of \cref{ssec:cont}\cref{eq:pmu} are
\begin{itemize}
\item if $Y$ is an irreducible $\BBQ G$-module, then $Y$ occurs in the graded
  module $H^*(M(\SA))$ in pairs of degrees that differ by $1$, and
\item if $\xi$ is a linear character of $G$, then $\sum_{k\geq 0} (-1)^k \dim
  H^k(M(\SA))^\xi =0$. \label{pp:it2}
\end{itemize}
It follows from the second assertion and \cref{pro:bries}\cref{it:bries2} that
\begin{equation}
  \label{eq:cn}
  \sum_{X\in \CX\setminus \{\Cent(\SA)\}} (-1)^{\cd X} \dim
  H^{\cd X}( M(\SA_X))^{\xi_X} + (-1)^{\rk \SA} \dim H^{\rk
    \SA}(M(\SA))^\xi=0 .  
\end{equation}
Thus, if $\dim H^{\cd X}( M(\SA_X))^{\xi_X}$ has been computed for $X$
different from $\Cent(\SA)$, then $\dim H^{\rk \SA}(M(\SA))^\xi$ is determined
by \cref{eq:cn}.

Together with Brieskorn's Lemma and \cref{thm:3}, the map $\partial$ can be
used to construct bases of $H^*(M(\SA))^G$ for irreducible reflection
pairs. The case when $\rk \SA=2$ is completed in the next proposition.

\begin{proposition}\label{pro:b2}
  Suppose that $(\SA,G)$ is an arrangement-group pair with $\rk \SA=2$.
  \begin{enumerate}
  \item If $G$ acts on $\SA$ with $a$ orbits, then the Poincar\'e polynomial
    of $(\SA,G)$ is
    \[
      P(\SA,G; t) = 1+ at+(a-1)t^2.
    \]
  \item If $\{H_1, \dots, H_a\}$ is a set of orbit representatives for the
    action of $G$ on $\SA$, then
    \[
      \{1\} \amalg \{e_G\cdot h_1 , \dots , e_G\cdot h_a\} \amalg \{ e_G\cdot
      h_1h_2, \dots, e_G\cdot h_1h_a\}
    \]
    is a basis of $H^*(M(\SA))^G$.
  \end{enumerate}
\end{proposition}

\begin{proof}
  For $k=0$, we've noted above that $\dim H^0(M(\SA))^G=1$ and so $\{1\}$ is a
  basis.

  If $k=1$, then it follows from \cref{ssec:osp2}\cref{eq:7} that $\dim
  H^1(M(\SA))^G=a$ and $\{e_G\cdot h_1, \dots, e_G\cdot h_a\}$ is a basis of
  $H^1(M(\SA))^G$.

  Suppose $k=2$. By \cref{ssec:cpp}\cref{eq:cn}, $\dim H^2(M(\SA))^G = a-1$,
  so $ P(\SA,G; t) = 1+ at+(a-1)t^2$. To complete the proof, it suffices to
  show that $\{e_G\cdot h_1h_2, \dots, e_G\cdot h_1h_a\}$ is linearly
  independent. Because $\partial \colon H^2(M(\SA))\to H^1(M(\SA))$ is
  injective and $G$-equivariant, it is enough to show that the span of
  $\{\partial(e_G\cdot h_1h_2), \dots, \partial(e_G\cdot h_1h_a)\}$ has
  dimension $a-1$. But this last assertion is clear because $\partial(e_G\cdot
  h_1h_i) = e_G\cdot h_i -e_G\cdot h_1$ for $2\leq i\leq a$.
\end{proof}

\section{Reduction to irreducible reflection pairs} \label{sec:red}

In this section $C\subseteq \GL(V)$ is a reflection coset, $\sigma\in C$ is a
fixed coset representative, $G_0= CC\inverse$, and $G_1= \langle C\rangle=
\langle G_0\sigma\rangle$. We sketch how the computation of
$H^*(M(\SA_i))^{G_j}$ for $i,j\in\{0,1\}$ can be reduced to the case when $C$
is irreducible.

\subsection{$\bm{G_1}$ acts reducibly}\label{ssec:red1}

First, suppose that $G_1$ does not act irreducibly on $V$. A standard argument
(see \cite[\S3E] {brouemallemichel:spetsesI}) shows that if $V'$ and $V''$ are
complementary, $G_1$-stable, proper subspaces of $V$, then $G_0$ is the
internal product of $G_0'= Z_{G_0}(V'')$ and $G_0''= Z_{G_0}(V')$ and $C$ is
isomorphic to the product reflection coset $C' \times C''$, where
\[
  C'= \big(Z_{G_0}(V'') \sigma\big) |_{V'}\subseteq
  \GL(V')\quad\text{and}\quad C''= \big(Z_{G_0}(V') \sigma
  \big)|_{V''}\subseteq \GL(V'').
\]
It follows from the $G_1$-invariance of the decomposition $V=V'+V''$ and the
K\"unneth theorem that
\[
  H^*(M(\SA_i))^{G_j} \cong H^*(M(\SA_i'))^{G_j'} \otimes
  H^*(M(\SA_i''))^{G_j''}
\]
for $i,j\in \{0,1\}$. This reduces the computation of $H^*(M(\SA_i))^{G_j}$ to
the case when $G_1$ acts irreducibly on $V$.

\subsection{$\bm{G_1}$ acts irreducibly}

Now suppose that $G_1$ acts irreducibly on $V$. Fix an irreducible
$G_0$-invariant subspace $V_0\subseteq V$. A standard argument (see the proof
of \cite[Prop.~6.9] {blairlehrer:cohomology}) shows that there is a minimal
positive integer, say $r$, such that $V= V_0+ \sigma V_0+\dots +\sigma^{r-1}
V_0$. Since $\sigma$ permutes the subspaces $\sigma^jV_0$ cyclically and $G_0$
acts completely reducibly on $V$, it follows from the minimality of $r$ that
$V$ is the internal direct sum of $V_0$, $\sigma V_0$, \dots,
$\sigma^{r-1}V_0$, and that $\sigma^r V_0=V_0$.

Define $Z_0= Z_{G_0}\big(\sigma V_0+\dots +\sigma^{r-1} V_0 \big)$. Then $Z_0$
is generated by the reflections in $G_0$ that fix $\sigma V_0+\dots
+\sigma^{r-1} V_0$ pointwise, $\sigma^r$ normalizes $Z_0$, and $G_0$ is the
internal direct product of the reflection subgroups $Z_0$, $\sigma
Z_0\sigma\inverse$, \dots, $\sigma^{r-1} Z_0 \sigma^{1-r}$. More precisely,
the map
\[
  m\colon Z_0^r \to G_0 \quad\text{defined by}\quad m(z_1, \dots, z_r)= z_1(\sigma
  z_2\sigma\inverse) \dotsm (\sigma^{r-1} z_r \sigma^{1-r})
\]
is a group isomorphism. Moreover, the restriction mapping from $Z_0$ to
$Z_0|_{V_0}$ is an isomorphism and the action of $G_0$ on $V_0$ factors
through $Z_0$, so $Z_0$ acts irreducibly on $V_0$ and $(Z_0, V_0)$ is an
irreducible reflection group. 

\begin{lemma}\label{lem:ags}
  With the preceding notation,
  \begin{equation}
    \label{eq:1ags}
    \SA(Z_0) = \{\, H\in \SA(G_0)\mid V_0\not\subseteq H\,\},
  \end{equation}
  \begin{equation}
    \label{eq:2ags}
    \SA_0 = \SA( Z_0) \amalg \sigma \big(\SA(Z_0) \big) \amalg \dotsm \amalg
    \sigma^{r-1} \big(\SA(Z_0) \big),
  \end{equation}
  and 
  \begin{equation}
    \label{eq:3ags}
    \SA_1 = \SA(\langle Z_0\sigma^{r} \rangle) \amalg \sigma \big(\SA(\langle
    Z_0\sigma^{r} \rangle) \big) \amalg \dotsm \amalg \sigma^{r-1}
    \big(\SA(\langle Z_0\sigma^{r} \rangle) \big),
  \end{equation}
  unless $r=2$ and $\dim V_0=1$.
  If $r=2$ and $\dim V_0=1$, then up to isomorphism, $G_1$ is equal to
  $G(q,1,2)$ and $G_0 \cong \mu_d^2$ is the group of diagonal matrices in
  $G(d,1,2)$, where $d$ is a divisor of $q$.
\end{lemma}

\begin{proof}
  First, if $H\in \SA(G_0)$, then either $\sum_{j=1}^{r-1}
  \sigma^jV_0\subseteq H$ or $V_0\subseteq H$, and not both. The equality in
  \cref{eq:1ags} follows immediately. The equality \cref{eq:2ags} follows from
  the decompositions $V=V_0+\sigma V_0+\dotsm+ \sigma^{r-1} V_0$ and $G_0= Z_0
  \cdot(\sigma Z_0\sigma\inverse) \dotsm(\sigma^{r-1} Z_0 \sigma^{1-r})$ using
  a similar argument. The equality \cref{eq:3ags} follows from analyzing the
  condition $\cd \Fix_V(g \sigma^{m})=1$ for $g\in G_0$ and $m>0$. One shows
  that (1) if $g\sigma^{rk}$ is a reflection, then
  \[
    \SA(G_0\sigma^{rk})= \SA(Z_0 \sigma^{rk}) \amalg \SA(\sigma Z_0
    \sigma\inverse \sigma^{rk}) \amalg \dotsm \amalg \SA(\sigma^{r-1} Z_0
    \sigma^{1-r}\sigma^{rk})
  \]
  and (2) if $m$ is not divisible by $r$, then $\SA(G_0 \sigma^{m})=
  \emptyset$ unless $r=2$ and $\dim V_0=1$. Finally, when $r=2$ and $\dim
  V_0=1$, the final assertion is easily verified.
\end{proof}

\subsection{}
Define
\[
  \Zbar_0= Z_0|_{V_0}, \quad \sigmarbar=\sigma^r|_{V_0},\quad \text{and}\quad
  \Cbar=\Zbar_0 \sigmarbar.
\]
Then $(\Zbar_0, V_0)$ is an irreducible reflection group, $\SA(\Zbar_0)= \{\,
V_0\cap H\mid H\in \SA(G_0),\ V_0\nsubseteq H\,\}$, and $\Cbar$ is an
irreducible reflection coset in $\GL(V_0)$. Obviously $\Cbar\, \Cbar\inverse =
\Zbar_0$. Define
\[
  Z_1= \langle Z_0 \sigma^r \rangle\quad\text{and}\quad \Zbar_1= \langle \Cbar
  \rangle= \langle \Zbar_0 \sigmarbar \rangle.
\]
To simplify the notation somewhat, set
\[
  \SAbar_i= \SA(\Zbar_i) \quad\text{for $i=0,1$.}
\]
We want to compute the space of invariants $H^*(M(\SA_i))^{G_j}$, for $i,j\in
\{0,1\}$, in terms of the spaces $H^*(M(\SAbar_i))^{\Zbar_j}$.

\subsection{}
Next, define
\[
  f\colon V_0^r\to V \quad\text{by}\quad f(v_1, v_2, \dots, v_r)= v_1+\sigma
  v_2+ \dots +\sigma^{r-1}v_r
\]
and
\[
  \widetilde\sigma\colon V_0^r\to V_0^r \quad\text{by}\quad
  \widetilde\sigma(v_1, \dots, v_r)= (\sigmarbar v_r, v_1, \dots,v_{r-1}),
\]
and let
\[
  \mbar\colon \Zbar_0^{\,r}\to G_0\quad\text{be the composition}\quad
  \Zbar_0^{\,r}\xrightarrow[\ \cong\ ]{} Z_0^r\xrightarrow[\ \cong\ ]{m} G_0,
\]
where the first isomorphism is the inverse of the $r$-fold product of the
restriction isomorphism $Z_0\xrightarrow{\ \cong\ } \Zbar_0$.
  
\begin{lemma}\label{lem:frho}
  \begin{enumerate}
  \item The maps $f$ and $\widetilde\sigma$ are vector space
    isomorphisms. \label{it:frho1}
  \item $f\circ \widetilde\sigma = \sigma \circ f$. \label{it:frho2}
  \item Via the isomorphism $\mbar$, \label{it:frho3}
    \begin{enumerate}
    \item $f$ intertwines the product action of $\Zbar_0^{\,r}$ on $V_0^r$
      with the action of $G_0$ on $V$ and
    \item $f$ intertwines the action of $\langle \Zbar_0^{\,r}
      \widetilde\sigma \rangle$ on $V_0^r$ with the action of $G_1$ on $V$.
    \end{enumerate}
  \item $f$ induces isomorphisms of arrangements $(\SAbar_i^{\,r} , V_0^r)
    \cong (\SA_i, V)$ for $i=0,1$. \label{it:frho4}
  \item $f$ restricts to homeomorphisms $M(\SAbar_i)^r \cong M(\SA_i)$ for
    $i=0,1$. \label{it:frho5}
  \end{enumerate}
\end{lemma}

\begin{proof}
  Statements \cref{it:frho1}, \cref{it:frho2}, and \cref{it:frho3} are direct
  computations, \cref{it:frho4} follows from \cref{lem:ags}, and
  \cref{it:frho5} follows from \cref{it:frho4}.
\end{proof}

\subsection{}
Using \cref{lem:frho} we may identify the groups $G_0\equiv \Zbar_0^{\,r}$ and
$G_1\equiv \langle \Cbar \rangle= \Zbar_1$, the reflection cosets $C\subseteq
\GL(V)$ and $\Cbar_0^{\,r} \subseteq \GL(V_0^r)$, and the reflection pairs
\begin{align*}
  \big(\SA_0, G_0  \big)%
  &\equiv \big(\SAbar_0^{\,r}, \Zbar_0^{\,r} \big),%
  &\big(\SA_0, G_1 \big)%
  &\equiv \big(\SAbar_0^{\,r}, \Zbar_1 \big) ,\\ 
  \big(\SA_1,G_0 \big)%
  &\equiv \big(\SA(\Zbar_1), \Zbar_0^{\,r}\big),%
  &\big(\SA_1, G_1\big)%
  &\equiv \big(\SA(\Zbar_1), \Zbar_1\big). 
\end{align*}

Since $\Cent(\SA(Z_0))= \sum_{j=1}^{r-1} \sigma^j V_0$, there is an
isomorphism $H^*(M(\SA(Z_i))) \cong H^*(M(\SAbar_i))$, for $i=0,1$, that on
elements we denote by $h\leftrightarrow \overline h$.  The isomorphism in
\cref{lem:frho}\cref{it:frho5}, together with the K\"unneth Theorem, shows
that there is a commutative diagram of isomorphisms of graded vector spaces
\[
  \xymatrix{H^*(M(\SA_i)) \ar[d]^{\sigma^*} \ar[r]^{f_i^*} &
    H^*(M(\SAbar_i))^{\otimes r} \ar[d]^{\widetilde\sigma^*} \\
    H^*(M(\SA_i)) \ar[r]^{f_i^*} & H^*(M(\SAbar_i))^{\otimes r} }
\]
such that for homogeneous elements $h_1, \dots, h_r\in H^*(M(\SA(Z_i)))$ we
have
\[
  \xymatrix{h_1\sigma(h_2) \dotsm \sigma^{r-1}(h_r) \ar@{|->}[d]^{\sigma^*}
    \ar@{|->}[r] ^{f_i^*} & \overline{h_1}\otimes \overline{h_2} \otimes
    \dotsm \otimes \overline{h_r} \ar@{|->}[d]^{\widetilde\sigma^*} \\
    \pm \sigma^{r}(h_r)\sigma(h_1) \dotsm \sigma^{r-1}(h_{r-1})
    \ar@{|->}[r]^{f_i^*} & \pm \sigma^r(\overline{h_r} ) \otimes
    \overline{h_1} \otimes \dotsm \otimes \overline{h_{r-1}} ,}
\]
where the sign is determined by the degrees of $h_1$, \dots, $h_r$. 

\begin{proposition}\label{pro:ij}
  For $i\in \{0,1\}$, the isomorphisms $f_i^*$ intertwine the action of $G_0$
  with the action of $\Zbar_0^{\,r}$, and the action of $G_1$ with the action
  of $\Zbar_1=\langle \Zbar_0^{\,r} \widetilde\sigma\rangle$, and induce
  isomorphisms
  \begin{align*}
    H^*(M(\SA_0))^{G_0}%
    &\cong \big(H^*( M(\SAbar_0) )^{\Zbar_0} \big)^{\otimes r},%
    & H^*(M(\SA_0))^{G_1}%
    &\cong \Big(\big(H^*( M(\SAbar_0) )^{\Zbar_0} \big)^{\otimes r}
      \Big)^{\widetilde\sigma^*} ,\\
    H^*(M(\SA_1))^{G_0}%
    &\cong \big(H^*( M(\SAbar_1) )^{\Zbar_0} \big)^{\otimes r}, %
    &H^*(M(\SA_1))^{G_1}%
    &\cong \Big(\big(H^*( M(\SAbar_1) )^{\Zbar_0} \big)^{\otimes r}
      \Big)^{\widetilde\sigma^*} .
  \end{align*}
\end{proposition}

\begin{proof}
  The first assertion follows by direct computation and the second follows
  from the first by taking invariants and using that $\Zbar_0^{\,r}$ acts
  componentwise on $r$-fold tensor products.
\end{proof}

\subsection{}

\cref{pro:ij} shows how to compute $H^*(M(\SA_i))^{G_j}$ for $i,j\in \{0,1\}$
starting from the reflection coset $\Cbar= \Zbar_0\sigmabar$. For
$H^*(M(\SA_0))^{G_0}$ and $H^*(M(\SA_1))^{G_0}$ this is clear from the
K\"unneth Theorem and the compatible factorizations of $\SA_0$, $\SA_1$, and
$G_0$. For example $H^*(M(\SA_0))^{G_0} \cong H^*(M(\SAbar_0))^{\Zbar_0}$. For
$H^*(M(\SA_0))^{G_1}$ and $H^*(M(\SA_1))^{G_1}$, choose a homogeneous basis of
$H^*(M(\SA(Z_i)))^{Z_0}$ consisting of eigenvectors for $(\sigma^r)^*$. Then
an explicit basis of $\big((H^*( M(\SAbar_i) )^{\Zbar_0} )^{\otimes r}
\big)^{\widetilde\sigma^*}$ is given by certain orbit sums (indexed by a set
of Lyndon words). Further details of the construction are omitted as this
basis does not play a role in the sequel.

\section{Irreducible reflection pairs} \label{sec:cla}

The next task is to compile a list of irreducible reflection pairs, $(\SA,G)$,
with the property that if $C$ is an irreducible reflection coset with $G_0=
CC\inverse$ and $G_1=\langle C\rangle$, then each $H^*(M(\SA_i))^{G_j}=
H^*(M(\SA))^{G}$ for some pair $(\SA, G)$ on the list. The first step is to
establish notation for the irreducible reflection groups, arrangements, and
cosets.

\subsection{Irreducible complex reflection groups and reflection
  types}\label{ssec:crg}

Because reflection groups are pairs, and not abstract groups, the
classification of reflection groups takes into account the underlying vector
space $V$. Following \cite[\S3]{orliksolomon:arrangements}, we say that two
reflection groups $(G,V)$ and $(G',V')$ have the same \emph{reflection type}
if there are isomorphisms $G\to G'$ and $V\to V'$ such that $(gv)'=g'v'$ for
$g\in G$ and $v\in V$. The reflection type of a non-trivial complex reflection
group is the concatenation of the reflection types of its non-trivial,
irreducible factors. The reflection types of the irreducible reflection groups
are given by the Shephard-Todd classification.

The notation we use for the irreducible complex reflection groups and their
reflection types is mostly consistent with
\cite[App.~C]{orlikterao:arrangements}. In particular, we use the Coxeter
label for symmetric groups and exceptional groups. If $(G, V)$ is an
irreducible complex reflection group, then with the Shephard-Todd grouping,
$(G,V)$ has the same reflection type as one of the following reflection
groups:
\begin{enumerate}
\item A Coxeter group of type $A_n$ for $n\geq 0$: $A_0$ is the trivial group
  acting on $V=\BBC$. For $n>0$, let $W_{n+1}=G(1,1,n+1)$ be the group of
  permutation matrices in $\GL_{n+1}(\BBC)$ and set $V=v_1^{\perp}$, where
  $v_1$ is the vector in $\BBC^n$ with all entries equal to $1$. Then
  $W_{n+1}$ stabilizes $V$ and $A_n$ is $(W_{n+1}|_V,V)$.
\item $G$ is in the infinite family of groups $G(r,p,n)$, acting on
  $V=\BBC^n$, where $r$, $p$, and $n$ are positive integers, $p$ divides $r$,
  $r\geq 2$, $n\geq 2$, and $(r,p,n)\ne (2,2,2)$. To enhance readability, in
  the tables below we use subscripts and write $G_{r,p,n}$ instead of
  $G(r,p,n)$.
  
\item $G$ is a group of roots of unity, $\mu_r$ for some $r> 2$, acting on
  $V=\BBC$.
\item $G$ is one of thirty-four exceptional groups labeled $G_{4}$, $G_5$,
  \dots, $G_{37}=E_8$, each acting on an explicitly defined complex vector
  space (see \cite[Ch.~8]{lehrertaylor:unitary}).
\end{enumerate}

\subsection{Irreducible reflection arrangements}\label{ssec:ra}

It is not that uncommon for non-isomorphic reflection groups to have the same
reflection arrangement. Parallel to the classification of irreducible
reflection groups, we have the following classification of irreducible
reflection arrangements.
\begin{enumerate}
\item We may realize $\SA(A_n)$ as the arrangement of $W_{n+1}$ acting on
  $\BBC^{n+1}$, which is the braid arrangement $\SB_{n+1}$.

  Notice that $\SB_{1}= \SA(A_0)$ is the empty arrangement and that $\SB_{2}=
  \SA(A_1)$ is the unique irreducible reflection arrangement with rank $1$.

\item Using the notation in \cite[Ch.~6]{orlikterao:arrangements}, for $n,r
  \geq 2$,
  \[
    \SA(G(r,r,n))= \SA_n^0(r), \quad\text{and} \quad \text{$\SA(G(r,p,n))=
      \SA_n(r)$ for $p<r$.}
  \]

\item For $r>1$, $\SA(\mu_r)=\SB_{2}$.
  
\item The arrangements of the exceptional complex reflection groups are
  denoted using script instead of roman characters, for example $\SG_4$ and
  $\SE_8$. There are equalities of arrangements,
  \[
    \SG_7= \SG_{10},\quad \SA_2(4)= \SG_8, \quad \SG_9 = \SG_{13},
    \quad\text{and}\quad \SG_{11}=\SG_{15}.
  \]
\end{enumerate}

\subsection{Irreducible reflection cosets}\label{ssec:rc}

Brou\'e, Malle, and Michel \cite{brouemallemichel:spetsesI} have classified
irreducible reflection cosets. If $C=G\sigma \subseteq \GL(V)$ is an
irreducible reflection coset with $\dim V>1$, then up to scalar multiples of
the identity, the pair $(G, \sigma)$ appears in the following list.
\begin{enumerate}
\item $G$ is irreducible and $\sigma$ is the identity linear
  transformation, \label{it:rc1}
\item $G=G(r,p,n)\subseteq \GL(\BBC^n)$ with $r>1$ and $\sigma =\varphi^q$,
  where $\varphi$ is the diagonal matrix with entries $(\omega_r, 1, \dots,
  1)$ and $0< q<r$,
\item $G=F_4$ and $\sigma=\gamma$ induces the graph automorphism of
  $F_4$, \label{it:rc3} 
\item $G=D_4=G(2,2,4)$ and $\sigma=\tau$ induces the triality automorphism of
  $D_4$, \label{it:rc4}
\item $G=G(3,3,3)$, embedded as a normal reflection subgroup in $G_{26}$ and
  $\sigma\in G_{26}$ is any element with order $4$ that normalizes
  $G$, \label{it:rc5}
\item $G=G_5$, or $G=G_7$, embedded as normal reflection subgroups of $G_{15}
  =\langle G_7, G_{14}\rangle$, and $\sigma=\rho_2$ is any reflection in
  $G_{14}$ with order $2$,
\item $G=G_5$, or $G=G_7$, with $G_5\subseteq G_7$ and $G_7$ embedded as a
  normal reflection subgroup of $G_{10}$, and $\sigma=\rho_4$ is any
  reflection in $G_{10}$ with order $4$, or
\item $G=G(4,2,2)$, embedded as a normal reflection subgroup in $G_6$, and
  $\sigma=\rho_3$ is any reflection in $G_6$ with order $3$. \label{it:rc8}
\end{enumerate}

\subsection{Irreducible reflection pairs}\label{ssec:rp}

As noted above, each reflection coset $C=G\sigma$ leads to a family of
irreducible reflection cosets $Cz=G\sigma z$, where $z$ runs over the group of
scalar transformations of the underlying vector space with finite order. For
each such $G\sigma z$ we want to compute
\begin{align*}
  &H^*(M(\SA(G)))^{G},%
  && H^*(M(\SA(G)))^{\langle G\sigma z\rangle}= H^*(M(\SA(G)))^{\langle
     G\sigma\rangle},\\ 
  &H^*(M(\SA(\langle G\sigma z \rangle)))^{G},%
  && H^*(M(\SA(\langle G\sigma z \rangle)))^{\langle G\sigma z\rangle}=
     H^*(M(\SA(\langle G\sigma z \rangle)))^{\langle G\sigma\rangle} .
\end{align*}
Thus, for each irreducible reflection coset $G\sigma$ as in
\cref{ssec:rc}\cref{it:rc1}-\cref{ssec:rc}\cref{it:rc8}, and scalar
transformation $z$, we need to compute $H^*(M(\SA))^{G}$, where $(\SA, G)$ is
one of the following four reflection pairs
\begin{equation}
		\label{eq:rpt}	
\begin{aligned}
	& (\SA(G), G), &&  (\SA(G), \langle G\sigma z\rangle),\\
	& (\SA(\langle G\sigma z\rangle), G), &&  (\SA(\langle G\sigma z\rangle), \langle G\sigma z\rangle).
\end{aligned}	
\end{equation}

To compile the list of pairs $(\SA, G)$ in \cref{eq:rpt} we first compute the
arrangements $\SA(\langle G\sigma z \rangle)$ and the groups $\langle G\sigma
\rangle$ for each of the reflection cosets $G\sigma$ in
\cref{ssec:rc}\cref{it:rc1}-\cref{ssec:rc}\cref{it:rc8}, and then organize the
results by the isomorphism type of $\SA$.

\begin{theorem}\label{thm:irp}
  Suppose $G\sigma$ is one of the reflection cosets in
  \cref{ssec:rc}\cref{it:rc1}-\cref{ssec:rc}\cref{it:rc8} and $z$ is a scalar
  transformation with finite order.
  \begin{enumerate}
  \item If $\rk G=2$, the arrangements $\SA(\langle G\sigma z\rangle)$ with
    $\SA(\langle G\sigma z\rangle) \ne \SA(G)$, and the groups $\langle
    G\sigma\rangle$, are given in \cref{tab:irp2} in the appendix.
  \item If $\rk G>2$, then $\SA(\langle G\sigma z\rangle) = \SA(G)$,
    unless $G=G(r,r,n)$ or $G=G_{25}$.
  \item If $G=G(r,r,n)$ with $n>2$ and $\SA(\langle G\varphi^q z\rangle)
    \ne \SA(G)$, then $\SA(\langle G\varphi^q z\rangle) = \SA_n(r)$.
  \item If $G=G_{25}$ and $\SA(\langle G z\rangle) \ne \SA(G)$, then
    $\SA(\langle G z\rangle) = \SG_{26}$.
  \end{enumerate}
\end{theorem}

\subsection{} \label{}

The proof of the theorem is a routine calculation, using Springer's theory of
regular elements, for each $\sigma$ in
\cref{ssec:rc}\cref{it:rc1}-\cref{ssec:rc}\cref{it:rc8}. The computations in
rank two are left to the reader. Here we sketch the main idea and work out
some examples to give the flavor of the calculations.

Suppose $G$ is any reflection group. Define a reflection coset, $C$, to be
\emph{regular} if $\SA(C)\nsubseteq \SA(CC\inverse)$. Equivalently, $C$ is
regular if and only if $C$ contains a reflection, $r$, such that
$\Fix(r)\notin \SA(CC\inverse)$. Clearly, if $C=G\sigma z$ and the index of
$G$ in $\langle G\sigma z\rangle$ is $m$, then $\SA( \langle G\sigma z\rangle)
= \bigcup_{k=1}^m \SA(G\sigma^kz^k)$. Thus, $\SA( \langle G\sigma z\rangle)
\ne \SA(G)$ if and only if $G\sigma^kz^k$ is a regular reflection coset for
some $k$ with $0< k<m$.

Suppose $G\sigma z$ is regular. Let $\zeta$ denote the eigenvalue of $z$ and
choose $g\in G$ so that $g\sigma z$ is a reflection with $\Fix(g\sigma
z)\notin \SA(G)$. Then $\zeta$ is a root of unity and $\Fix(g\sigma z)$ is
equal to the $\zeta\inverse$-eigenspace of $g\sigma$.  Because $\Fix(g\sigma
z)\notin \SA(G)$, the $\zeta\inverse$-eigenspace of $g\sigma$ contains a
regular vector in $V$. The proof of the next theorem is a straightforward
application of \cite[Thm.~6.4]{springer:regular}.

\begin{theorem}\label{thm:spr}
  Suppose that $G\sigma z$ is a regular reflection coset, the degrees of
  $G$ are $d_1$, \dots, $d_n$, and $f_1$, \dots, $f_n$ are basic polynomial
  invariants of $G$ that are also eigenfunctions for $\sigma$. Say $\deg
  f_i=d_i$ and $\sigma \cdot f_i= \epsilon_i f_i$ for $1\leq i\leq n$. Choose
  $g\in G$ so that $g\sigma z$ is a reflection.  Then
  \begin{enumerate}
  \item the eigenvalues of the reflection $g\sigma z$ are $\epsilon_1\inverse
    \zeta^{d_1}$, \dots, $\epsilon_n\inverse \zeta^{d_n}$ (so exactly one of
    $\epsilon_1\inverse \zeta^{d_1}$, \dots, $\epsilon_n\inverse \zeta^{d_n}$
    is not equal to one) and
  \item if $h\in G$ and $h\sigma z$ is a reflection, then $\Fix(g\sigma z)$
    and $\Fix(h\sigma z)$ are in the same $G$-orbit.
  \end{enumerate}
  In particular, $\SA(G\sigma z)$ is a single $G$-orbit of hyperplanes in $V$,
  which is disjoint from $\SA(G)$.
\end{theorem}

The eigenvalues $\epsilon_i$ are computed in
\cite[App.~D.5]{lehrertaylor:unitary}. 

\subsection{}
Using \cite[Thm.~12.23]{lehrertaylor:unitary} the theorem can be rephrased to
include a characterization of regular reflection cosets. This stronger result
is not needed in full generality, but it is convenient to use the special case
when $\sigma$ is the identity.

For an integer $d$, let $a(d)$ denote the number of degrees of $G$ that are
divisible by $d$. Similarly, define $a^*(d)$ to be the number of codegrees of
$G$ that are divisible by $d$. We say that $d$ is a \emph{regular reflection
  number of $G$} if $a(d)=a^*(d)=n-1$. The degrees and codegrees of the
exceptional reflection groups are given in
\cite[App.~D.2]{lehrertaylor:unitary}, from which the regular reflection
numbers can easily be computed.

Suppose $z$ has order $d$. It follows from \cite[Thm.~12.23]
{lehrertaylor:unitary} that $Gz$ is a regular reflection coset if and only if
$d$ is a regular reflection number for $G$. If so, then $\SA(Gz)$ is a single
$G$-orbit of hyperplanes, which is obviously disjoint from $\SA(G)$.

Suppose $G$ is irreducible. If $Gz$ and $Gz'$ are two reflection cosets, then
$Gz=Gz'$ if and only if $z\inverse z'$ is in the center of $G$, and if so,
then $\SA(Gz)=\SA(Gz')$. Thus scalar representatives of reflection cosets, and
in particular regular reflection cosets, are determined modulo the center of
$G$.

It turns out that if the rank of $G$ is greater than two, $G$ has regular
reflection numbers if and only if $G=G(r,r,n)$ or $G=G_{25}$. These two cases
are discussed in more detail below. In rank two, the only irreducible
reflection groups with no regular reflection numbers are the groups $G(r,p,2)$
with $p$ even, $G_7$, $G_{11}$, $G_{15}$, and $G_{19}$.

\subsection{Examples}\label{ssec:ex}

In these examples, as well as in the appendix, for a positive integer, $r$,
let $z_r\in \GL(\BBC^2)$ denote the scalar matrix with eigenvalue $\omega_r$.

A first example is the group $G_{25}$. The degrees of $G_{25}$ are $6$, $9$,
and $12$ and the codegrees are $0$, $3$, $6$. Thus, the regular reflection
numbers of $G_{25}$ are $2$ and $6$. Moreover, $\gcd(6,9,12)=3$, so the center of
$G_{25}$ has order three and hence $G_{25}z_2= G_{25}z_6$. It follows that
there is a unique regular reflection coset, namely $G_{25}z_2$. It is
well-known that $\langle G_{25} z_2\rangle =G_{26}$ (see \cite[Thm.~8.42]
{lehrertaylor:unitary}).

\subsection{}
Next suppose that $G=G(r,p,n)$ with $n>2$.

The degrees of $G$ are $r$, $2r$, \dots, $(n-1)r$, $nr/p$ and the codegrees
are $0$, $r$, \dots, $(n-1)r$. Thus $G(r,p,n)$ has no regular reflection
numbers if $p\ne r$ and the reflection numbers of $G(r,r,n)$ are the
divisors of $r$ that do not divide $n$. If $d$ divides $r$ and not $n$, then
$\langle G(r,r,n)z_d\rangle \subseteq G(r,1,n)$ and $\SA( \langle
G(r,r,n)z_d\rangle) \ne \SA(G(r,r,n))$, so $\SA(\langle G(r,r,n)z_d\rangle )=
\SA_n(r)$.

Now consider $\sigma=\varphi^q$ with $0<q<r$. If $p\ne r$, then
\begin{equation}
  \label{eq:10}
  \SA_n(r) = \SA(G) \subseteq \SA(\langle G \varphi^q z\rangle) \subseteq
  \SA(\langle G(r,1,n) \varphi^q z\rangle) =\SA(\langle G(r,1,n) z\rangle) =
  \SA_n(r),
\end{equation}
where the last equality holds because $G(r,1,n)$ has no regular reflection
numbers. Hence $\SA(\langle G \varphi^q z\rangle) = \SA(G)$ for every scalar
transformation $z$. If $p=r$ and $\SA(\langle G \varphi^q z\rangle) \ne
\SA(G)$, then the computation in \cref{eq:10} becomes
\[
  \SA_n^0(r) = \SA(G) \subsetneq \SA(\langle G \varphi^q z\rangle) \subseteq
  \SA(\langle G(r,1,n) \varphi^q z\rangle) = \SA_n(r),
\]
from which it again follows that $\SA(\langle G \varphi^q z\rangle) = \SA_n(r)$.

\subsection{}
As a final example, suppose $G=D_4=G(2,2,4)$ and that conjugation by
$\sigma=\tau$ induces the triality automorphism of $G$ as in
\cref{ssec:rc}\cref{it:rc4}. Consider a reflection coset $G\tau^q z$. The
degrees of $D_4$ are $2$, $4$, $4$, $6$, and a basic set of polynomial
invariants of $G$ may be chosen so that the eigenvalues of $\tau$ on these
invariants are $1$, $\omega_3$, $\omega_3\inverse$ and $1$, respectively. Just
suppose that $G\tau z$ is regular and choose a reflection $g\tau z\in G\tau
z$. By \cref{thm:spr} the eigenvalues of $g\tau z$ are $\zeta^{2}$,
$\omega_3\inverse \zeta^4$, $\omega_3 \zeta^4$, and $\zeta^6$, where $\zeta$
is the eigenvalue of $z$. It is never true that exactly three of $\zeta^{2}$,
$\omega_3\inverse \zeta^4$, $\omega_3 \zeta^4$, and $\zeta^6$ are equal to
one. This is a contradiction, so $G\tau z$ is not a regular reflection coset
for any $z$.

Sorting the computations in \cref{thm:irp} by reflection arrangement we obtain
the classification of irreducible reflection pairs in the next corollary.

\begin{corollary}\label{cor:irp}
  Suppose $C$ is an irreducible reflection coset with $G_0= C C\inverse$ and
  $G_1 = \langle C \rangle$. Then for $i,j\in \{0,1\}$, one of the following
  statements holds.
  \begin{enumerate}
  \item $H^*( M(G_i) )^{G_j} = H^*( M( \SA(\Gtilde) ))^{\Gtilde}$, where
    $\Gtilde$ is irreducible.
  \item The rank of $G_0$ is equal to two and $H^*( M(G_i) )^{G_j} = H^*( M(
    \SA) )^{G}$, where $(\SA, G)$ is one of the pairs in \cref{tab:rk2} in the
    appendix.  The table also records a reflection coset, $C$, that gives rise
    to the reflection pair $(\SA, G)$.
  \item The rank of $G_0$ is greater than two and $H^*( M(G_i) )^{G_j} = H^*(
    M(\SA) )^{G}$, where $(\SA, G)$ is one of the following
    pairs: \label{it:irp3}
    \begin{enumerate}
    \item $(\SA,G)= (\SA_n(r), G(r,r,n))$ or $(\SA,G)= (\SA_n^0(r),
      G(r,p,n))$, where $p<r$. Both pairs arise from the coset
      $G(r,r,n)\varphi^p$.
    \item $(\SA,G)= (\SF_4, \langle F_4\gamma \rangle)$. This pair arises from
      the coset $F_4 \gamma$. \label{it:irp3f4}
    \item $(\SA,G)= (\SA_4^0(2), \langle D_4\tau \rangle)$. This pair arises
      from the coset $D_4 \tau$. \label{it:irp3d4}
    \item $(\SA,G)= (\SA_3^0(3), \langle G(3,3,3) \sigma \rangle)$. This pair
      arises from the coset $G(3,3,3) \sigma$. \label{it:irp3g3}
    \end{enumerate}
  \end{enumerate}
\end{corollary}

\subsection{Families of irreducible reflection pairs}\label{ssec:firp}

For the purpose of computing the spaces $H^*(M(\SA))^G$, the collection of
irreducible reflection pairs naturally divides into five families.
\begin{enumerate}
\item A first family is the collection of all rank two irreducible reflection
  pairs. These reflection pairs are listed in \cref{tab:rk2}, along with
  reflection cosets that give rise to them.

  When $(\SA,G)$ has rank two, the Poincar\'e polynomial $P(\SA,G; t)$, and a
  basis of $H^*(M(\SA))^G$, can be computed using \cref{pro:b2}, once the
  orbits of $G$ on $\SA$ are known. These orbits are straightforward to
  compute for each rank two, irreducible reflection pair. In the rest of this
  paper we take $P(\SA,G; t)$ and $H^*(M(\SA))^G$ as known for rank two
  reflection pairs.

\item A second family is the set of pairs $(\SA(A_n), A_n)$ for $n>2$. It was
  observed above that if $n>1$, then $H^*(M(A_{n-1}))^{A_{n-1}} \cong
  H^*(M(\SB_{n}))^{W_{n}}$.

  Brieskorn \cite{brieskorn:tresses} has shown that $P(\SB_{n},W_{n}; t) =
  1+t$. Because $H^1(M(\SB_{n}))^{W_{n}}$ affords the permutation
  representation of $W_n$ on $\SB_n$, a basis of $H^*(M(\SB_{n}))^{W_{n}}$ is
  given by $1\in H^0(M(\SB_{n}))^{W_{n}}$ and the orbit sum of the
  Orlik-Solomon generator, $h$, in $H^1(M(\SB_{n}))^{W_{n}}$, where $H$ is any
  hyperplane in $\SB_n$.
  
\item Third is the collection of pairs $(\SA(G), G)$, where $G$ is an
  exceptional, irreducible, reflection group with rank greater than two.  As a
  matter of terminology, in the rest of this paper we call these pairs
  \emph{primitive.}

  The computation of $P(\SA,G;t)$ for primitive reflection pairs proceeds by
  recursion. This argument is an extension of that given by Brieskorn for the
  exceptional Coxeter groups.

\item Next are the three irreducible reflection pairs in
  \cref{cor:irp}\cref{it:irp3f4}, \cref{it:irp3d4}, \cref{it:irp3g3}, which we
  call \emph{very exceptional.}

  The computations for the very exceptional reflection pairs is most easily
  dealt with case-by-case, using results from the other families.

\item Finally, the irreducible reflection pairs, $(\SA,G)$, where $\SA\in
  \{\SA_n^0(r), \SA_n(r)\}$ and $G\in \{\, G(r,r,n), G(r,p,n) \,\}$ with
  $r>1$, $p$ is a proper divisor of $r$, and $n>2$. We call these pairs
  \emph{imprimitive.}

  The computation of $P(\SA,G;t)$ for imprimitive reflection pairs proceeds by
  induction on $n$. Again, this argument is an extension of that given by
  Brieskorn for $(\SB_n, W_n)$.
\end{enumerate}

\section{Invariants and Poincar\'e polynomials for irreducible
  reflection pairs} \label{sec:thm3}

\subsection{}

Suppose $(\SA,G)$ is an irreducible reflection pair. Then $\SA=\SA(\Gtilde)$,
where either (1) $\Gtilde=G$ or (2) $\Gtilde \ne G$ and there is an
irreducible reflection coset, $C$, in \cref{tab:rk2} or
\cref{cor:irp}\cref{it:irp3}, such that $\Gtilde\in \{ CC\inverse, \langle
C\rangle\}$.

Recall from \cref{pro:bries}\cref{it:bries2} that
\[
  H^k(M(\SA))^G \cong \bigoplus_{X\in \CX(\SA,G)_k} H^{k} ( M(\SA_X))
  ^{N_G(X)} .
\]
Clearly, to compute $\dim H^k(M(\SA))^G$ it is sufficient to determine the set
of $X\in \CX(\SA,G)_k$ such that $\dim H^{k} ( M(\SA_X)) ^{N_G(X)} \ne 0$ and
compute $\dim H^{k} ( M(\SA_X)) ^{N_G(X)}$ for each such $X$.

Define
\[
  \CX(\SA,G)_k^{\tdi} = \{\, X\in \CX(\SA,G)_k \mid \dim H^{k} (
  M(\SA_X))^{N_G(X)} \ne 0\,\}
\]
(\textrm{tdi} for ``top degree invariants''). In practice it turns out that
$\CX(\SA,G)_k^{\tdi}$ has at most two elements and that if $X\in
\CX(\SA,G)_k^{\tdi}$, then $\dim H^{k} ( M(\SA_X)) ^{N_G(X)}$ is at most two,
except when $\SA=\SA_2(2r)$, $G=G(r,p,2)$, and $p$ is even (in which case the
dimension is equal to $3$).

\subsection{Orbit representatives for primitive and imprimitive reflection
  pairs}

Suppose $(\SA,G)$ is an irreducible reflection pair and $X\in L(\SA)$.  Then
$Z_{\Gtilde}(X)$ acts faithfully on $X^{\perp}$. We may identify
$Z_{\Gtilde}(X)$ with $Z_{\Gtilde}(X)|_{X^{\perp}}$ and thus consider
$Z_{\Gtilde}(X)$ as a subgroup of $\GL(X^{\perp})$ generated by
reflections. Then $(Z_{\Gtilde}(X), X^{\perp})$ is an \emph{effective} reflection
group, that is, no irreducible factor of the reflection type of
$(Z_{\Gtilde}(X), X^{\perp})$ is equal to $A_0$.

Let $\CT$ denote the set of reflection types of effective complex reflection
groups and define
\[
  \rt\colon L(\SA)\to \CT\quad\text{by} \quad \rt(X)=\text{``the reflection
    type of $(Z_{\Gtilde}(X), X^{\perp})$''}.
\]
It is clear that if $g\in G$, then $(Z_{\Gtilde}(X), X^{\perp})$ and
$(Z_{\Gtilde}(gX), (gX)^{\perp})$ have the same reflection type, so $\rt$ is
constant on $G$-orbits. Define the \emph{reflection type} of a $G$-orbit to be
$\rt(X)$ for any $X$ in the orbit.

\subsection{}\label{ssec:pot}

Suppose that $(\SA, G) =(\SA(G),G)$ is a primitive reflection pair. Orlik and
Solomon have shown that the map $\rt$ is very close to being an injection when
restricted to $\CX(\SA,G)$ and have given a classification of the orbits of
$G$ in $L(\SA(G))$ based on the reflection types of pointwise stabilizers of
orbit representatives (see \cite[App.~C] {orlikterao:arrangements}). This
classification is used in \cref{thm:3} and its proof as follows: Let
$\CT(\SA,G)$ denote the set of orbit types in the Orlik-Solomon
classification. For each $T\in \CT(\SA,G)$ choose $X_T\in L(\SA)$ with
$\rt(X_T)=T$. We may take $\CX(\SA,G)= \{\, X_T\mid T\in \CT(\SA,G)
\,\}$. Then $\CT(\SA,G)$ indexes both $\CX=\CX(\SA,G)$ and $L(\SA)/G$. Define
\[
  \CT(\SA,G)_k= \{\, T\in \CT(\SA,G)\mid X_T\in \CX_k\,\} \text{ and }
  \CT(\SA,G)^{\tdi}= \{\, T\in \CT(\SA,G)\mid X_T\in \CX^{\tdi}\,\}.
\]

The rule that maps an orbit in $L(\SA(G))$ to the isomorphism type of the
pointwise stabilizer of an orbit representative is not always one-to-one. For
primitive reflection pairs it turns out that if $T$ occurs as the reflection
type of the pointwise stabilizer for more than one orbit, then $T\notin
\CT(\SA,G)^{\tdi}$, so this ambiguity does not affect the computations in this
paper. For example, the group $G_{27}$ has seven orbits in $L(\SG_{27})$,
\[
  \CT(\SG_{27}, G_{27}) = \{ A_0, A_1, A_2', A_2'', B_2, I_2(5), G_{27} \},
\]
so $\CT(\SG_{27}, G_{27})_0 = \{A_0\}$, $\CT(\SG_{27}, G_{27})_1 = \{A_1\}$,
$\CT(\SG_{27}, G_{27})_2 = \{A_2', A_2'', B_2, I_2(5)\}$, and $\CT(\SG_{27},
G_{27})_3 = \{G_{27}\}$. It turns out that $\CT(\SG_{27}, G_{27})^{\tdi} = \{
A_0, A_1, B_2, G_{27} \}$.

It follows from \cref{thm:3} that $|\CT(\SG, G)^{\tdi}|$ tends to be
relatively small when compared with the cardinality of $|\CT(\SG, G)|$. For
example, $|\CT(\SG_{34}, G_{34})|=24$ and $|\CT(\SE_{8}, E_{8})|=41$, but it
turns out that $|\CT(\SG_{34}, G_{34})^{\tdi}|= |\CT(\SE_{8},
E_{8})^{\tdi}|=4$.

\subsection{}

Next suppose that $(\SA, G)$ is imprimitive. It is not unusual for several
orbits to have the same reflection type. As explained in \cref{ssec:orb}, it
is more natural for computations to label the orbits of $G$ in $L(\SA)$ by
partitions, specifically, partitions of $m$ with $0\leq m\leq n$, together
with a second parameter in case $m=n$.

For a partition, $\lambda$, $X_\lambda$ denotes a $G$-orbit representative
with type $\lambda$. It turns out that the restriction of $\rt$ to
$\CX^{\tdi}$ is injective, hence a bijection. Consequently, we may identify
$\CX(\SA,G)^{\tdi}$ and $\CT(\SA,G)^{\tdi}$, and label these orbits by
partitions, $\lambda$, or reflection types, $T$, without ambiguity. With the
notation in \cref{ssec:orb}, the partitions indexing orbit representatives in
$\CX(\SA,G)^{\tdi}$ and the corresponding reflection types in
$\CT(\SA,G)^{\tdi}$ for imprimitive reflection pairs are given in the top part
in \cref{tab:thm32}.

\begin{theorem}\label{thm:3}
  Suppose $(\SA, G)$ is an irreducible reflection pair with rank greater than
  two.
  \begin{enumerate}
  \item \label{it:thm33} If $(\SA,G)$ is primitive or very exceptional, then
    $\dim H^{\rk T}( M(\SA_{X_T}))^{N_G(X_T)}=1$ for
    $T\in\CT(\SA,G)^{tdi}$. The sets $\CT(\SA, G)^{tdi}$ and Poincar\'e
    polynomials $P(\SA(G),G;t)$ are given in \cref{tab:thm33}.
    
  \item \label{it:thm32} If $(\SA,G)$ is imprimitive, then the partitions
    $\lambda\in \CX(\SA,G)^{\tdi}$, reflection types $T\in \CT(\SA,G)^{\tdi}$,
    and $\dim H^{\rk T}( M(\SA_{X_T}))^{N_G(X_T)}$ for $T\in \CT_G^{\tdi}$,
    are given in \cref{tab:thm32}.
  \end{enumerate}
\end{theorem}

\begin{table}
  \renewcommand{\arraystretch}{1.2} \centering
  \caption{$(\SA, G)$ primitive or very exceptional}
  \begin{tabular} {>{$}l<{$}>{$}l<{$} @{\hspace {2em}}>{$}l<{$} @{\hspace
    {2em}} >{$}l<{$}}
    \toprule \addlinespace
    \rk&(\SA, G)&\CT(\SA,G)^{\tdi}& P(\SA, G;t) \\
    \addlinespace \toprule \addlinespace
    3& (\SG_{25},G_{25}) &A_0, C_3& 1+t\\
    3& (\SA_3^0(3), \langle G_{3,3,3} \sigma \rangle)&A_0,A_1  &\\
    4& (\SG_{32},G_{32})& A_0, C_3 \\
    6& (\SE_6 ,E_{6}) & A_0, A_1 \\

    \addlinespace \midrule \addlinespace
    3& (\CH_3 ,H_{3})& A_0,A_1,A_1^2,H_3 & 1+t+t^{n-1} + t^n\\
    3& (\SG_{24},G_{24})&A_0,A_1, B_2,G_{24} \\
    3& (\SG_{27},G_{27})&A_0 ,A_1,B_2,G_{27} \\
    4& (\SG_{29},G_{29})&A_0,A_1, B_3,G_{29} \\
    4& (\CH_4 ,H_{4})&A_0,A_1, H_3,H_4\\
    4& (\SG_{31},G_{31})&A_0,A_1, G_{4,2,3} ,G_{31}\\
    4& (\SF_4, \langle F_4\gamma \rangle)&A_0,A_1,B_3, F_4  &\\
    5& (\SG_{33},G_{33})&A_0,A_1, D_4, G_{33}\\
    4& (\SA_4^0(2), \langle D_4\tau \rangle)&A_0,A_1,A_1^3, D_4&\\ 
    6& (\SG_{34},G_{34})&A_0,A_1, G_{33} ,G_{34} \\
    7& (\SE_7 ,E_{7})&A_0,A_1, D_6,E_7 \\
    8& (\SE_8 ,E_{8})&A_0,A_1, E_7 ,E_8 \\

    \addlinespace \midrule \addlinespace
    3& (\SG_{26},G_{26})&A_0,A_1,C_3, A_1C_3 , G_{3,1,2} ,G_{26}%
                                  & 1+2t+2t^2+t^3\\ 
    4& (\SF_4 ,F_{4})&A_0,A_1,\tilde A_1, A_1 \tilde A_1,B_2,C_3,B_3, F_4
                                  &1+2t+2t^2+2t^3+t^4\\
    \addlinespace \bottomrule \addlinespace
  \end{tabular}
  \label{tab:thm33}
\end{table}

\begin{table}[htb]
  \caption{$(\SA,G)$ imprimitive, $G=G(r,p,n)$ with $r\geq2$, $n\geq 3$,
    and $p\leq r$}
  \begin{minipage}{\linewidth}
    \footnotesize \centering
    \begin{tabular} {l cc >{$}c<{$} >{$}c<{$} >{$}c<{$} >{$}c<{$} >{$}c<{$}
      >{$}c<{$} >{$}c<{$} >{$}c<{$}}
      \toprule \addlinespace
      $\deg$&&& k=0 %
      & \multicolumn{2}{c}{$k=1$}%
      & \multicolumn{2}{c}{$2\leq k\leq n-2$}
      & \multicolumn{2}{c}{$k=n-1$} & k=n \\  
      \cmidrule(lr){1-1} \cmidrule(lr){4-4} \cmidrule(lr){5-6}
      \cmidrule(lr){7-8} \cmidrule(lr){9-10} \cmidrule(lr){11-11}
      $\CX^{\tdi}$&& $\lambda$&(1^n)%
      & (2\,1^{n-2})& (1^{n-1})%
      & (2\,1^{n-k-1})&(1^{n-k})& (2)& (1)&  \emptyset\\
      $\CT^{\tdi}$ &&$T$&A_0&A_1&G_{r,p,1}%
      & G_{r,p,k-1} A_1  & G_{r,p,k}%
                  & G_{r,p,n-2} A_1 & G_{r,p,n-1} &G_{r,p,n}\\ 
      \addlinespace \midrule \addlinespace
      \multirow{2}{*}{$\SA_n(r)$} &&$p,n$ even& 1 & 1 & 1& 1& 1& 2 & 1 &2\\
      \addlinespace
            &&else & 1 & 1 & 1& 1& 1& 1 & 1 &1\\
      \addlinespace  \midrule \addlinespace
      \multirow{2}{*}{$\SA_n^0(r)$} 
            & & $p,n$ even& 1 & 1 & -& -& -& 1 & - & 1\\
            &&else & 1 & 1 & -& -& -& - & - &-\\
      \addlinespace \bottomrule \addlinespace
    \end{tabular}
  \end{minipage}
  \label{tab:thm32}
\end{table}

A first consequence of the theorem is that the polynomials $P(\SA(G),G; t)$
have only four possible patterns.

\begin{corollary}\label{cor:1}
  Suppose that $G$ is a non-trivial, irreducible, complex reflection group
  with rank at least two.
  \begin{enumerate}
  \item If $G$ is one of either $G(r,r,n)$, where either $n$ or $r$ is odd;
    $G_4$; $G_8$; $G_{12}$; $G_{16}$; $G_{20}$; $G_{22}$; $G_{25}$; $G_{32}$;
    or $E_6$, then
    \[
      P(\SA(G), G;t)= 1+t.
    \]

  \item If $G$ is one of either $G(r,r,n)$, where both $n$ and $r$ are even;
    $G_{23}$; $G_{24}$; $G_{27}$; $G_{29}$; $G_{30}$; $G_{31}$; $G_{33}$;
    $G_{34}$; $E_7$; or $E_8$, then
    \[
      P(\SA(G), G;t)= 1+t + t^{n-1} +t^n.
    \]

  \item If $G$ is one of either $G(r,p,n)$, where $p<r$ and either $n$ or $p$
    is odd; $G_5$; $G_6$; $G_9$; $G_{10}$; $G_{13}$; $G_{14}$; $G_{17}$;
    $G_{18}$; $G_{21}$; $G_{26}$; or $G_{28}$; then
    \[
      P(\SA(G), G;t)= 1+2t +\dotsm + 2t^{n-1} +t^n.
    \]

  \item If $G$ is one of either $G(r,p,n)$, where $p<r$ and both $n$ and $p$
    are even; $G_7$; $G_{11}$; $G_{15}$; or $G_{19}$, then
    \[
      P(\SA(G), G;t)= 1+2t +\dotsm + 2t^{n-2} +3t^{n-1} +2t^n.
    \]
  \end{enumerate}
\end{corollary}

\subsection{}
If $n=2$, then $1+t+t^{n-1}+t^n = 1+2t+t^2$. Hence the formula in the second
statement in the corollary applies to the groups $G(r,p,2)$ with $p$ odd and
$p<r$, $G_5$, $G_6$, $G_9$, $G_{10}$, $G_{13}$, $G_{14}$, $G_{17}$, $G_{18}$,
and $G_{21}$ as well.

The next corollary is the analog of \cref{cor:1} for the reflection pairs
$(\SA,G(r,p,n))$, where $\SA\in \{\SA_n^0(r), \SA_n(r)\}$.

\begin{corollary}\label{cor:2}
  Suppose $n$, $r$, and $p$ are positive integers such that $p$ divides $r$
  and $n\geq 2$. Then
  \begin{align*}
    P(\SA_n^0(r) , G(r,p,n); t)
    &= \begin{cases}
      1+t &\text{if $p$ or $n$ is odd,} \\
      1+t + t^{n-1} +t^n&\text{if $p$ and $n$ are even,} \\
    \end{cases} \\ \intertext{and} P(\SA_n(r) , G(r,p,n); t) 
    &= \begin{cases} 1+2t + \dotsm + 2t^{n-2} + 2t^{n-1} +t^n
      &\text{if $p$ or $n$ is odd,} \\
      1+2t + \dotsm + 2t^{n-2} + 3t^{n-1} +2t^n %
      &\text{if $p$ and $n$ are even.} \\
    \end{cases} \\
  \end{align*}
\end{corollary}

Another consequence of the theorem is the curious dichotomy implicit in the
next corollary.

\begin{corollary}
  Suppose $(\SA,G)$ is an irreducible reflection pair with rank at least
  two. Then the following are equivalent:
  \begin{enumerate}
  \item $H^{\rk G}(M(\SA))^G = 0$, that is $(\SA,G)$ does not have top degree
    invariants.
  \item $P(\SA, G;t) = 1+t$.
  \end{enumerate}
\end{corollary}

\section{Proof of \cref{thm:3}}\label{sec:prf} 

In this section we prove \cref{thm:3}, first for primitive reflection pairs,
then for imprimitive pairs. The proof for very exceptional reflection pairs
uses the results of \cref{thm:4} (for primitive and imprimitive reflection
pairs) and is given in \cref{ssec:thm3ve}.

\subsection{Preparing for the proof}\label{ssec:prep} 

Before giving the proof of \cref{thm:3} we state some preliminary results and
work out an example. In the proof of \cref{thm:3}\cref{it:thm33} extensive
use is made of the presentations given by Brou\'e, Malle, and Rouquier. As a
matter of terminology, the diagrams in \cite[Tab.~1-4]
{brouemallerouquier:complex} that encode these presentations are called
\emph{BMR diagrams}.

The proofs of the next two lemmas are immediate. 

\begin{lemma}
  \label{lem:nott1}
  Suppose that $(\SA, G)$ and $(\SA', G')$ are reflection pairs. If $(\SA, G)$
  does not have top degree invariants, then neither does $(\SA\times \SA',
  G\times G')$.
\end{lemma}


\begin{lemma}\label{lem:nott2}
  If $(\SA, G)$ is an arrangement-group pair, $X\in L(\SA)$, and there is a
  $g\in N_G(X)$ that acts on $H^{\cd X} ( M(\SA_{X})) ^{Z_G(X)}$ as $-1$, then
  $H^{\cd X} ( M(\SA_{X}) )^{N_G(X)} =0$.
\end{lemma}

In practice, the element $g$ frequently arises as either the long word, or an
analog of the long word (in the non-Coxeter case), in a subgroup that contains
$Z_G(X)$ that is suggested by the BMR diagram of $G$.

\subsection{}

Suppose $(\SA,G)$ is an arrangement-group pair with underlying vector space
$V$ such that $Z(G)$, the center of $G$, acts on $V$ as scalars. Then $Z(G)$
acts trivially on both $\SA$ and $H^*(M(\SA))$, and $Z(G)\subseteq N_G(X)$ for
$X\in L(\SA)$.

\begin{lemma}\label{lem:zfac}
  Suppose $(\SA,G)$ is an arrangement-group pair such that $Z(G)$ acts on $V$
  as scalars and $X\in L(\SA)$. If $N_G(X) = Z_G(X)\cdot Z(G)$, then $N_G(X)$
  acts trivially on $H^{\cd X}(M(\SA_X)) ^{Z_G(X)}$. In particular, if $X\ne
  0$ and $|G|= |Z_G(X)| \cdot |Z(G)|\cdot |G X|$, then $N_G(X)$ acts trivially
  on $H^{\cd X} ( M(\SA_{X})) ^{Z_G(X)}$.
\end{lemma}

\begin{proof}
  The first assertion is clear. With the hypotheses in the second assertion,
  $Z_G(X)$ intersects $Z(G)$ trivially, so it follows from the equality $|G|=
  |Z_G(X)| \cdot |Z(G)|\cdot |G X|$ that $N_G(X) = Z_G(X)\cdot Z(G)$.
\end{proof}

\subsection{Example: the group $\bm {G_{32}}$} \label{ssec:G32}

To illustrate the ideas in the proof of \cref{thm:3}, in this subsection we
prove the theorem for the reflection pair $(\SG_{32}, G_{32})$, assuming that
the theorem has been proved for groups with rank less than four. Set
\[
  G=G_{32}, \quad \SG=\SA(G)=\SG_{32}, \quad\text{and}\quad \CT= \CT (\SG, G).
\]
The orbits in $L(\SG)/G$ are indexed by reflection types, and as in
\cref{ssec:pot},
\[
  \CT = \{ A_0, \mu_3 ,G_{4}, \mu_3^2, \mu_3G_4, G_{25}, G_{32} \} .
\]
By \cref{pro:bries}\cref{it:bries3} it is sufficient to compute $\dim H^{\rk
  T}(M(\SG_{X_T}))^{N(T)}$, for $T\in \CT$, where to minimize parentheses and
subscripts we have set $N(T)=N_G(X_T)$. This computation is done recursively
by rank, building up from $\CT_0$ and $\CT_1$.

It was observed in \cref{ssec:tdi} that $\CT_0 \cup \CT_1 \subseteq
\CT(\SG,G)^{\tdi}$. Because $\CT_1 =\{\mu_3\}$, $G$ acts transitively on $\SG$,
and so $\dim H^{\rk T}(M(\SG_{X_T}))^{N(T)}=1$ for $T\in \CT_0 \cup \CT_1$.

We are assuming that we have already shown that the reflection pair $(\SG_4,
G_4)$ does not have top degree invariants, so $G_4\notin \CT^{\tdi}$. It
follows from \cref{lem:nott1} that $\mu_3G_4\notin \CT^{\tdi}$.

Consider the orbit indexed by the reflection type $T=\mu_3^2$. Referring to
the BMR diagram in \cite[Tab.~1]{brouemallerouquier:complex} for $G_{32}$,
define
\[
  X_T= \Fix(\langle s,u \rangle) \quad\text{and}\quad g=(stu)^2.
\]
Then $Z_G(X_T)= \langle s,u \rangle$ and by \cref{pro:b2}, $H^2(M(\SG_{X_T}
))^{Z_G(X_T)}$ is one-dimensional with basis $\{h_sh_u\}$.  Using the fact
that $g^2$ is a generator of the center of $G_{25}$ (see the row indexed by
$G_{25}$ in \cite[Tab.~1] {brouemallerouquier:complex}), it is straightforward
to check that conjugation by $g$ interchanges $s$ and $u$, so $g$ acts as $-1$
on $H^2(M(\SG_{X_T} ))^{Z_G(X_T)}$, and so $H^2(M(\SG_{X_T}) )^{N(T)}=0$ by
\cref{lem:nott2}. Therefore $\mu_3^2\notin \CT^{\tdi}$.

We are assuming that we have already shown that the reflection pair
$(\SG_{25}, G_{25})$ does not have high degree invariants, so $G_{25}\notin
\CT^{\tdi}$.

Finally, set $Z_0=Z_G(X_{A_0})$, $N_0=N_G(X_{A_0})$, $Z_1=Z_G(X_{\mu_3})$, and
$N_1=N_G(X_{\mu_3})$. Then by \cref{ssec:cpp}\cref{eq:cn},
\begin{align*}
  0 & =  \dim H^{0}(M(\SA(Z_{0})))^{N_{0}} -\dim H^{1}(M(\SA(Z_{1})))^{N_{1}}+
      (-1)^{4} \dim H^4(M(\SG))^G\\
    & =  1-1+ \dim H^4(M(\SG))^G.
\end{align*}
Therefore, $H^4(M(\SG))^G= 0$. It follows that $\CT^{\tdi}=\{A_0, \mu_3\}$ and
$\dim H^{\rk T}(M(\SG_{X_T}) )^{N(T)}=1$ for $T\in \CT^{\tdi}$.

\subsection{Primitive reflection pairs}\label{ssec:exc}

In this subsection $(\SG,G)$ is a primitive reflection pair. Thus $G$ is an
exceptional complex reflection group with rank at least three, $\SG=\SA(G)$,
and $\CT(\SG,G)$ indexes $L(\SG)/G$, as described in \cref{ssec:pot}. To
simplify the notation and help keep the subscripts under control, set
\[
  \CT=\CT(\SG,G) \quad \CX=\CX(\SG,G)= \{\, X_T\mid T\in \CT\,\}, 
  \quad\text{and}\quad \SG_T=\SG_{X_T}= \SA(Z_G(X_T)) .
\]
We compute $\dim H^{\rk T} (M(\SG_T))^{N_T}$ for $T\in \CT$ recursively,
following the same line of reasoning as for $G_{32}$ in \cref{ssec:G32}.

To begin, as noted in \cref{ssec:tdi}, $\CT_0 \cup \CT_1\subseteq \CT^{\tdi}$
and $\dim H^{\rk T} (M(\SG_T))^{N_T}=1$ for $T\in \CT_0 \cup \CT_1$.

Next consider orbits indexed by reflection types $T$ with $1< \rk T <\rk G$.
Again, it turns out that most of the time it follows from \cref{lem:nott1}
that $\dim H^{\rk T} (M(\SG_{T})) ^{N_T} =0$.  \Cref{tab:3} contains the pairs
$(G,T)$ where $G$ is one of the groups under consideration and $T$ is the
reflection type indexing an orbit of $G$ in $L(\SG)$ to which
\cref{lem:nott1} does not apply. In all but four of these cases, the dimension
of $H^{\rk T} (M(\SG_{T})) ^{N_T}$ can be determined using \cref{lem:nott2} or
\cref{lem:zfac}. The entries in the table should be interpreted as follows:
\begin{itemize}
\item Dots are included to improve readability.
\item The reflection type $T=A_1^k$ denotes $k$ copies of $A_1$ for some
  $k\geq 2$. As indicated by ``$-1$'', \cref{lem:nott2} applies in all cases.
\item An entry of ``$-1$'' indicates that \cref{lem:nott2} applies, so $\dim
  H^{\rk T} (M(\SG_{T})) ^{N_T} =0$.
\item An entry of ``$1$'' indicates that \cref{lem:zfac} applies and $N_T$
  acts on $H^{\rk T} (M(\SG_{T})) ^{Z_T}$ trivially. In all cases it turns
  out, by recursion, that $\dim H^{\rk T} (M(\SG_{T})) ^{N_T} =1$.
\item An entry of ``$1b$'' indicates that the corresponding parabolic subgroup
  is ``bulky.'' It follows from \cite{pfeifferroehrle:special} that for these
  two pairs, $N_T$ centralizes $Z_T$, and so by recursion,
  \[
    \dim H^{\rk T} (M(\SG_{T})) ^{N_T} =\dim H^{\rk T} (M(\SG_{T})) ^{Z_T} =
    \dim H^{\rk T} (M(\SG_T)) ^{Z_T} =1.
  \]
\item The entry ``\cref{ssec:g31}'' for the pair $(G_{31}, G(4,2,2))$ is a
  reference to that subsection, where it is shown that $N_T$ acts on $H^{2}
  (M(\SG_{T})) ^{Z_T}$ as the symmetric group $W_3$ acting on its reflection
  representation, and so $\dim H^{2} (M(\SG_{T})) ^{N_T}=0$.

\item The entry ``\cref{ssec:g33}'' for the pair $(G_{33}, D_4)$, is a
  reference to that subsection, where it is shown that $N_T$ acts trivially on
  $H^{4} (M(\SG_{T})) ^{Z_T}$, and so by recursion we have $\dim H^{4}
  (M(\SG_{T})) ^{N_T}=1$.
\end{itemize}

Finally, $\dim H^{\rk G}(M(\SG))^G$ can be computed in each case from the rows
in \cref{tab:3} using the values of $|\SG/G|$ given in the tables in
\cite[App.~C]{orlikterao:arrangements} and \cref{ssec:cpp}\cref{eq:cn}. For
example, for the group $G=G_{34}$ we have $|\SG/G|=1$ and $1-1+1-\dim H^{\rk
  G}(M(\SG))^G=0$, so $\dim H^{\rk G}(M(\SG))^G =1$.

\begin{table}[htb] \tiny
  \renewcommand{\arraystretch}{1.2}
  \centering
  \caption{Orbit reflection types not covered by \cref{lem:nott1}}
  \begin{tabular} {>{$}l<{$} @{\hspace {2em}}>{$}c<{$} >{$}c<{$} >{$}c<{$}
      >{$}c<{$} >{$}c<{$} >{$}c<{$} >{$}c<{$} >{$}c<{$} >{$}c<{$} >{$}c<{$}
      >{$}c<{$} >{$}c<{$} >{$}c<{$} >{$}c<{$} >{$}c<{$} >{$}c<{$}}
    \toprule \addlinespace
    G& \multicolumn{16}{c}{$T$} \\ \cmidrule(lr{2.25em}){1-1}
    \cmidrule{2-17} %
     &A_1^k &A_1\tilde A_1%
            &A_1\mu_3 &\mu_3^2%
                      &B_2 & G_{3,1,2}%
                           & G_{4,2,2} &B_3 %
                                       &\mu_3 &G_{4,2,3}%
                                              &H_3&D_4&G_{33}&A_1D_4&D_6
                                                                    &E_7 \\ 
    \addlinespace \toprule  \addlinespace 
    H_3&-1&.&.&.&.&.&.&.&.&.&.&.&.&.&.&. \\    
    G_{24}&.&.&.&.&1 &.&.&.&.&.&.&.&.&.&.&. \\
    G_{25}&.&.&.&-1&.&.&.&.&.&.&.&.&.&.&.&. \\
    G_{26}&.&.&1&.&.&1&.&.&.&.&.&.&.&.&.&. \\
    G_{27}&.&.&.&.&1&.&.&.&.&.&.&.&.&.&.&. \\
    F_4&-1&1b&.&.&1b&.&.&1&1&.&.&.&.&.&.&. \\
    G_{29}&-1&.&.&.&-1&.&.&1&.&.&.&.&.&.&.&.\\
    H_4&-1&.&.&.&.&.&.&.&.&.&1&.&.&.&.&. \\
    G_{31}&-1&.&.&.&.&.&\text{\cref{ssec:g31}}&.&.&1&.&.&.&.&.&. \\
    G_{32}&.&.&.&-1&.&.&.&.&.&.&.&.&.&.&.&. \\
    G_{33}&-1&.&.&.&.&.&.&.&.&.&.&\text{\cref{ssec:g33}}&.&.&.&. \\
    G_{34}&-1&.&.&.&.&.&.&.&.&.&.&-1&1&.&.&.\\
    E_6&-1&.&.&.&.&.&.&.&.&.&.&-1&.&.&.&. \\
    E_7&-1&.&.&.&.&.&.&.&.&.&.&-1&.&-1&1&. \\
    E_8&-1&.&.&.&.&.&.&.&.&.&.&-1&.&-1&-1&1 \\
    \addlinespace \bottomrule \addlinespace
  \end{tabular}
  \label{tab:3}
\end{table}

\subsection{The orbit with reflection type $\bm{G(4,2,2)}$ in
  $\bm{G_{31}}$} \label{ssec:g31}

In this subsection $G=G_{31}$, $\SG=\SG_{31}$, $T=G(4,2,2)$, and we sketch the
calculation that $H^2(M(\SG_{T}))^{N_{T}} = 0$.

Referring to the BMR diagram in \cite[Tab.~3]{brouemallerouquier:complex} for
$G_{31}$, take
\[
  Z_T= \langle s,t,u\rangle
\]
so $Z_T$ is a reflection group with reflection type $G(4,2,2)$.

We need a description of $H^2(M(\SG_{T}))$. To that end, set
\[
  s'= tst=usu,\quad t'=sts=utu,\quad\text{and}\quad u'=tut=sus.
\]
Then
\[
  \SG_{T}= \{H_s, H_t, H_u, H_{s'}, H_{t'}, H_{u'} \},
\]
where $H_x=\Fix(x)$, and $\{ h_sh_{s'}, h_sh_{t}, h_sh_{t'}, h_sh_{u},
h_sh_{u'}\}$ is a basis of $H^2(M(\SG_{T}))$.

Define
\[
  z=stu,\quad w_1= svtvsv\quad\text{and}\quad w_2 =uwtwuw,
\]
where $w$ is as in the BMR diagram in
\cite[Tab.~3]{brouemallerouquier:complex} for $G_{31}$.  Proofs of the
following assertions are straightforward and are omitted.
\begin{itemize}
\item The center of $Z_T$ is a cyclic group of order $4$ generated by $z$, and
  $\{ 1, s, t, u\}$ is a cross section of $\langle z\rangle$ in $Z_T$.
\item The set $\{e_{Z_T}\cdot h_th_s, e_{Z_T}\cdot h_th_u \}$ is a basis of
  $H^2(M(\SG_{T}))^{Z_T}$.
\item Conjugation by $w_1$ maps $s$ to $tz^2$, $t$ to $s$, and fixes
  $u$. Hence $w_1 \cdot h_s=h_t$, $w_1 \cdot h_t=h_s$, and $w_1 \cdot
  h_u=h_u$. Similarly, $w_2 \cdot h_u=h_t$, $w_2 \cdot h_t=h_u$, and $w_2
  \cdot h_s=h_s$.
\item The subgroup of $N_T/Z_T$ generated by $w_1Z_T$ and $w_2Z_T$ is
  isomorphic to the symmetric group $W_3$ and acts on $H^2(M(\SG_{T}))^{Z_T}$
  as the reflection representation of $W_3$.
\end{itemize}

It follows from the last assertion that $H^2(M(\SG_{T}))^{N_T} =0$, as
claimed.

\subsection{The orbit with reflection type $D_4$ in $G_{33}$} \label{ssec:g33}

In this subsection $G=G_{33}$, $\SG=\SG_{33}$, $T=D_4$, and we sketch an
argument that shows that $N_T$ acts trivially on $H^4(M(\SG_{T}))^{Z_{T}}$.

Referring to the first BMR diagram in
\cite[Tab.~4]{brouemallerouquier:complex} for $G_{33}$, define
\[
  z_1= stwtst, \quad z_2= wuvuwu, \quad\text{and} \quad g=swvz_2z_1.
\]
Then conjugation by $z_1$ interchanges $s$ and $w$ and fixes $t$, and
conjugation by $z_2$ interchanges $w$ and $v$ and fixes $u$.  Therefore,
conjugation by $g$ maps $s$ to $v$, $v$ to $w$, and $w$ to $s$.

Now set $r=tut$. It is straightforward to check that the subgroup generated by
$s$, $w$, $v$, and $r$ is a Coxeter group of type $D_4$ with $\{s,w, v,r\}$ as
a Coxeter generating set. Take
\[
  Z_{T} =\langle s,w,v,r\rangle.
\]
A short calculation using the relation $utwutw=wutwut$ shows that $grg\inverse
= r$. Hence, $g\in N_{T}$ and conjugation by $g$ acts on $Z_{T}$ as the
triality automorphism of $D_4$ with respect to the Coxeter generating set
$\{s,w,v,r\}$. It is straightforward to check that $|\langle g\rangle| = 6 =
|N_{T}:Z_{T}|$ and that $\langle g\rangle\cap Z_{T}$ is the trivial subgroup,
so $\langle g\rangle$ is a complement to $Z_{T}$ in $N_{T}$.

It follows from \cite{felderveselov:coxeter} and the isomorphism
\[
  H^4(M(\SA_4^0(2)))^{D_4} \cong H^4(M(\SG_{T}))^{Z_{T}}
\]
that $H^4(M(\SG_{T}))^{Z_{T}}$ is one-dimensional with basis vector $e_{Z_T}
h_sh_wh_vh_r$. Then
\[
  ge_{Z_T}\cdot h_sh_wh_vh_r = e_{Z_T} g\cdot h_sh_wh_vh_r= e_{Z_T} \cdot
  h_vh_sh_wh_r= e_{Z_T}\cdot h_sh_wh_vh_r,
\]
and it follows that $N_{T}$ acts trivially on $H^4(M(\SG_{T}))^{Z_{T}}$, as
asserted.

This completes the proof of \cref{thm:3}\cref{it:thm33} for primitive
reflection pairs.

\subsection{Imprimitive reflection pairs}\label{ssec:grpn}

We now turn to the proof of \cref{thm:3}\cref{it:thm32}. To continue we need
to establish some more notation related to the groups $G(r,p,n)$ and the
arrangements $\SA_n^0(r)$ and $\SA_n(r)$. When considering one of these
arrangements, the underlying vector space is always $V=\BBC^n$ and unless
otherwise noted, $r,n\geq 2$.

For $1\leq i\leq n$ let $x_i\in V^*$ denote projection on the $i\th$
coordinate. Suppose $1\leq i \neq j\leq n$ and $\zeta\in \mu_r$. Define
hyperplanes
\[
  H_{i}=\ker x_i= \{\, v\in V \mid v_i=0\,\} \quad\text{and}\quad
  H_{ij}(\zeta)=\ker (x_i-\zeta x_j)= \{\, v\in V\mid v_i= \zeta v_j\,\},
\]
where $v_k=x_k(v)$. Then
\[
  \SA_{n}^0(r)= \{\, H_{ij} (\zeta) \mid 1\leq i< j\leq n,\ \zeta\in \mu_r
  \,\} =\SA(G(r,r,n))
\]
and
\[
  \SA_{n}(r)=  \SA_{n}^0(r)
  \amalg \{\, H_i\mid 1\leq i\leq n\,\}= \SA(G(r,p,n)) 
\]
when $p<r$. By convention, $\SA_{1}^0(r) =\emptyset$ is the empty
arrangement. A special case is $r=1$: $\SA_n^0(1)=\SB_n = \SA(W_n)$ is the
$n\th$ braid arrangement.

Let $D_{r,n}$ denote the subgroup of $\GL_n(\BBC)$ consisting of diagonal
matrices with entries in $\mu_r$. The determinant map restricts to a
surjective group homomorphism from $D_{r,n}$ to $\mu_r$. If $p$ is a divisor
of $r$, define $D(r,p,n)$ to be the preimage of the subgroup $\mu_{r/p}$. Then
$G(r,p,n)=W_nD(r,p,n)$. With this notation we have $W_n =G(1,1,n) \subseteq
G(r,p,n)$.

\subsection{Orbit representatives, centralizers, and
  subarrangements} \label{ssec:orb}

Suppose $n\geq 2$ and let $\lambda=(\ell_1, \dots, \ell_a)$ be a partition of
$m$ for an integer $m$ with $0\leq m\leq n$. Set $\bar \ell_0=n-m$ and for
$i>0$ let $\bar \ell_i=n-m+\ell_1+ \dots + \ell_i$ denote the sum of $n-m$ and
the $i\th$ partial sum of $\lambda$. Let $\{e_1, \dots, e_n\}$ be the standard
basis of $\BBC^n$ and for $1\leq i\leq a$ define $b_i= e_{\bar \ell_{i-1}+1} +
\dots +e_{\bar \ell_i}$. Finally, define
\[
  X_\lambda= \spn \{\, b_1, \dots, b_a\,\}.
\]
Then $X_\lambda$ is an $a$-dimensional subspace of $\BBC^n$. Notice that
$X_{\emptyset}=0$ and $X_{(1^n)} =V$, and that $\cd X_\lambda = n-a$.

For a positive integer $p$, define
\[
\delta_{\lambda,p} = \gcd(p, \ell_1, \dots, \ell_a)
\]
and define $d_0\in \GL(\BBC^n)$ to be the diagonal matrix with entries
$\omega_r, 1, \dots, 1$.

It follows from the computations in \cite[\S6.4]{orlikterao:arrangements} that
\[
  \CX\big(\SA_n^0(r), G(r,p,n)\big)= \big\{\, X_\lambda \mid \lambda \vdash
  m,\ 0\leq m\leq n-2\,\big\} \amalg \big\{\, d_0^u X_\lambda \mid \lambda
  \vdash n,\ 0\leq u< \delta_{p,\lambda} \,\big\}
\]
and
\[
  \CX\big(\SA_n(r), G(r,p,n)\big)= \big\{\, X_\lambda \mid \lambda \vdash m,\
  0\leq m\leq n-1\,\big\} \amalg \big\{\, d_0^u X_\lambda \mid \lambda \vdash
  n,\ 0\leq u< \delta_{p,\lambda} \,\big\}
\]
are sets of orbit representatives for the action of $G(r,p,n)$ on $\SA_n^0(r)$
and $\SA_n(r)$, respectively.

\subsection{} \label{ssec:xz}

With $\lambda =(\ell_1, \dots, \ell_a)$ as above, it is straightforward to
check that $Z_{G(r,p,n)}(X_\lambda)$ is the ``diagonal'' subgroup of
$G(r,p,n)$ isomorphic to $G(r,p,n-m) \times W_{\ell_1} \times \dots \times
W_{\ell_a}$, consisting of block diagonal matrices with blocks $g$, $w_1$,
\dots, $w_a$, where $g\in G(r,p,n-m)$ and $w_i\in W_{\ell_i}$. Then
\begin{equation}
  \label{eq:zlamfac}
  Z_{G(r,p,n)}(X_\lambda) = G(r,p,n-m) \boxtimes W_{\ell_1} \boxtimes \cdots
  \boxtimes W_{\ell_a} \subseteq G(r,p,n), 
\end{equation}
where $\boxtimes$ denotes block diagonal direct sum. In particular,
$Z_{G(r,p,n)} (X_\lambda)$ is a reflection subgroup of $\GL(\BBC^n)$.

It is also straightforward to check that
\begin{equation*}
  \label{eq:a0lamfac}
  \SA_n^0(r)_{X_\lambda}\cong \SA_{n-m}^0(r)\times \SB_{\ell_1} \times \dotsm
  \times \SB_{\ell_a}
\end{equation*}
and
\begin{equation}
  \label{eq:anlamfac}
  \SA_n(r)_{X_\lambda} \cong \SA_{n-m}(r)\times \SB_{\ell_1} \times \dotsm
  \times \SB_{\ell_a} .
\end{equation}

Recall that $W_1$ is the trivial group acting on $\BBC$ and that
$\SB_1=\emptyset$. Thus, if $\lambda$ has $b$ parts of size greater than $1$,
then the reflection type of $Z_{G(r,p,n)}(X_\lambda)$ is $G(r,p,n-m)
A_{\ell_1-1} \cdots A_{\ell_b-1}$.

\subsection{The reflection pairs $\bm{(\SA_n(r), G(r,p,n))}$}

In the rest of this section we take
\[
  G=G(r,p,n), \quad \SA=\SA_n(r), \quad\text{and}\quad \CX=\CX(\SA, G),
\]
where $p\leq r$ and $r,n\geq 2$. To prove \cref{thm:3}\cref{it:thm32} for the
reflection pair $(\SA,G)$ we need to compute $\dim H^{\cd X} (M(\SA_{X}
))^{N_G(X)}$ for $X\in \CX$. The proof is by induction on $n$. The induction
begins with the base cases $n=2$ and $n=3$. The case $n=2$ follows from
\cref{pro:b2}. Details of the argument when $n=3$ are left to the reader. For
the inductive step, we assume that $n\geq 4$, and so $G(r,p,n-2)$ has rank at
least two.

The strategy is to first show that $\dim H^{\cd X} (M(\SA_{X} ))^{N_G(X)}=0$
unless $X=X_\lambda$, where $\lambda$ is either the empty partition,
corresponding to the reflection type $G$, or a partition with largest part at
most $2$, and at most one part equal to $2$. For these latter partitions,
$\dim H^{\cd X_\lambda} (M(\SA_{X_\lambda} ))^{N_G(X_\lambda)}$ is computed
using induction and some elbow grease. Finally, once $\dim H^{\cd X}
(M(\SA_{X} ))^{N_G(X)}$ has been computed for $X\ne X_\emptyset$ (or $T\ne
G$), then $\dim H^{\cd X_{\emptyset}} (M(\SA_{X_{\emptyset}}
))^{N_G(X_\emptyset)} =\dim H^{n}(M(\SA))^{G}$ can be determined using
\cref{ssec:cpp} \cref{eq:cn}.

The proof of \cref{thm:3}\cref{it:thm32} when $\SA=\SA_n^0(r)$ is similar, but
uses no new ideas and is technically less complicated, so is omitted.

\subsection{}
First notice that if $\lambda$ is a partition of $n$ and $d$ is a diagonal
matrix in $G(r,1,n)$, then $\dim H^{\cd dX_\lambda} (M(\SA_{dX_\lambda}
))^{N_G(dX_\lambda)}= \dim H^{\cd X_\lambda} (M(\SA_{X_\lambda}
))^{N_G(X_\lambda)}$, so it is sufficient to compute $\dim H^{\cd X_\lambda}
(M(\SA_{X_\lambda} ))^{N_G(X_\lambda)}$ for $X_\lambda\in \CX$.

\subsection{}\label{ssec:aaa}
Let $\lambda=(\ell_1, \dots, \ell_a)$ be a partition of $m$, with $0\leq m\leq
n$, that has $b$ parts of size greater than $1$. Set
\[
  V_0= \spn\{e_1, \dots, e_{n-a+b}\}\quad\text{and}\quad V_1=
  \spn\{e_{n-a+b+1}, \dots, e_{n}\}
\]
and let $p\colon \BBC^n\to V_0$ be the projection. Then
\[
  V_0 =\BBC^{n-m} \times \BBC^{\ell_1} \times \dotsm\times \BBC^{\ell_b},
  \quad p(\SA_{X_\lambda}) = \SA_{n-m}(r) \times \SB_{\ell_1} \times \dotsm
  \times \SB_{\ell_b},
\]
and
\[
  p(Z_{G} (X_\lambda))= G(r,p,n-m) \times W_{\ell_1} \times \dotsm \times
  W_{\ell_b} ,
\]
where as in \cref{ssec:ess} $p(Z_{G} (X_\lambda))$ denotes the image of $Z_{G}
(X_\lambda)$ in $\GL(V_0)$. Also as in \cref{ssec:ess}, the K\"unneth Theorem
induces an isomorphism
\begin{multline}
  \label{eq:3g}
  H^{n-a}(M(\SA_{X_\lambda})) ^{Z_G(X_\lambda)} \\
  \cong H^{n-m}(M( \SA_{n-m}(r))) ^{G(r,p,n-m)} \otimes H^{\ell_1-1} (M(
  \SB_{\ell_1})) ^{W_{\ell_1}} \otimes \dotsm \otimes H^{\ell_b-1} (M(
  \SB_{\ell_b})) ^{W_{\ell_b}}.
\end{multline}

\begin{lemma}
  If $H^{\cd X_\lambda} (M(\SA_{X_\lambda} ))^{N_G(X_\lambda)} \ne 0$, then
  either $\lambda=\emptyset$ or $\lambda =(2^b, 1^{a-b})$, where $b\leq 1$ and
  $a+b=m$.
\end{lemma}

\begin{proof}
  Brieskorn \cite{brieskorn:tresses} has shown that $P(\SB_{\ell}, W_{\ell};
  t)=1+t$ for $\ell\geq2$. Thus, if $\ell_1>2$, then $(\SB_{\ell_1},
  W_{\ell_1})$ does not have top degree invariants, so it follows from
  \cref{ssec:aaa}\cref{eq:3g} and \cref{lem:nott1} that
  $H^{n-a}(M(\SA_{X_\lambda})) ^{Z_G(X_\lambda)}=0$. Therefore, if $\dim
  H^{\cd X_\lambda} (M(\SA_{X_\lambda} ))^{N_G(X_\lambda)} \ne 0$, then
  $\lambda =(2^b, 1^{a-b})$, where $a+b=m$.

  Suppose $\lambda =(2^b, 1^{a-b})$ with $a+b=m$ and $b>1$. For the moment,
  let $s_i\in W_n$ denote the reflection that interchanges the basis vectors
  $e_i$ and $e_{i+1}$ and define $\Htilde_1=\Fix(s_{n-m+1})$,
  $\Htilde_3=\Fix(s_{n-m+3})$, and $g$ to be the longest element in the
  Coxeter group of type $A_3$ generated by $\{s_{n-m+1}, s_{n-m+2}, s_{n-m+3}
  \}$. Then $\Htilde_1, \Htilde_3\in \SA_{X_\lambda}$, $g\in
  N_{G}(X_\lambda)$, and $g$ acts on $\SA_{X_\lambda}$ by interchanging
  $\Htilde_1$ and $\Htilde_3$, and fixing $\SA_{X_\lambda} \setminus
  \{\Htilde_1,\Htilde_3\}$ elementwise. By \cref{eq:3g}, $H^{\cd X_\lambda}
  (M(\SA_{X_\lambda} ))^{N_G(X_\lambda)}$ is spanned by products of the form
  $x \htilde_1 \htilde_3 y$, where $x\in H^{n-m}(M( \SA_{n-m}(r)))
  ^{G(r,p,n-m)}$, $\htilde_1$ and $\htilde_3$ are the Orlik-Solomon generators
  corresponding to $\Htilde_1$ and $\Htilde_3$, respectively, and $y\in H^{1}
  (M(\SB_{1}))^{\otimes (a-b)}$, where we have used the identification
  \[
    Z_{G}(X_\lambda) = G(r,p,n-m) \boxtimes\, \underbrace{W_{2} \boxtimes
      \cdots \boxtimes W_2}_{b}\, \boxtimes\, \underbrace{W_1\boxtimes \cdots
      \boxtimes W_1}_{a-b} \subseteq G
  \]
  in \cref{ssec:xz}\cref{eq:zlamfac}. Then $g\cdot x \htilde_1 \htilde_3 y = x
  \htilde_3 \htilde_1 y = -x \htilde_1 \htilde_3 y$, and so $H^{\cd X_\lambda}
  (M(\SA_{X_\lambda} ))^{N_G(X_\lambda)} = 0$, by \cref{lem:nott2}.
\end{proof}

\subsection{}\label{ssec:et}

In a further effort to keep levels of subscripts under control, in the rest of
this section, for a partition, $\lambda$, set
\[
  \cd \lambda= \cd X_\lambda \quad Z(\lambda)=Z_G(X_\lambda),
  \quad\text{and}\quad N(\lambda)=N_G(X_\lambda).
\]

There are $2n-1$ non-empty partitions of the form $(2^b\, 1^{a-b})$ with
$0\leq a+b\leq n$, and $b \le 1$, namely
\[
  \eta_k = (1^{n-k}),\text{ for $0\leq k\leq n-1$,}\quad\text{and} \quad
  \tau_k =(2\, 1^{n-k-1}) ,\text{ for $1\leq k\leq n-1$,}
\]
where the notation is chosen so that $\cd \eta_k=\cd \tau_k=k$.

Clearly, the only partitions of $n$ in this collection are $\eta_0$ and
$\tau_1$. Since $\delta_{p,\eta_0} = \delta_{p,\tau_1} =1$, each of $\eta_0$
and $\tau_1$ indexes a unique orbit in $L(\SA)$.

Next, corresponding to codimension $0$ and $1$, there are three special cases:
$\eta_0=(1^n)$, $\eta_1=(1^{n-1})$, and $\tau_1=(2\, 1^{n-2})$.
\begin{itemize}
\item $X_{\eta_0}=X_{(1^n)} = V$, $Z(\eta_0)$ is the trivial subgroup, with
  reflection type $A_0$, and we have seen that $\dim H^0(M(\SA_V)) =1$.
\item $X_{\eta_1}=X_{(1^{n-1})}$ is the hyperplane $H_n\in \BBC^n$, $Z(\eta_1)
  \cong G(r,p,1)$ and has reflection type $\mu_p$ if $p<r$ and $A_0$ if $p=r$,
  and clearly $\dim H^1(M(\SA_{H_n}))^{Z(\eta_1)} =1$.

  When $p=r$, the reflection type $A_0$ indexes two orbits, namely those
  indexed by $\eta_0=(1^n)$ and $\eta_1=(1^{n-1})$. In order to distinguish
  these two orbits we denote the reflection type of $Z_{G(r,r,n)}
  (X_{(1^{n-1})})$ by $G(r,r,1)$.
\item $X_{\tau_1}=X_{(2\, 1^{n-2})}$ is the hyperplane $H_{1,2}\subset
  \BBC^n$, $Z(\tau_1)$ has reflection type $A_1$, and it follows from
  \cref{ssec:ess} that $\dim H^1(M(\SA_{H_{1,2}}))^{Z(\tau_1)}=1$.
\end{itemize}
For $k>1$, the reflection type of the pointwise stabilizer $Z(\eta_k)$ is
$G(r,p,k)$ and the reflection type of the pointwise stabilizer $Z(\tau_k)$ is
$G(r,p,k-1) A_1$.

\begin{lemma}\label{lem:a0}
  Suppose that $2\leq k\leq n-1$.
  \begin{enumerate}
  \item \label{it:lema01} $H^{k}( M(\SA_{\eta_k} ))^{N(\eta_k)} \cong H^{k} (
    M(\SA_k(r)))^{G(r,1,k)}$.
  \item \label{it:lema02} If $k<n-1$, then $H^{k}( M(\SA_{\tau_k}
    ))^{N(\tau_k)} \cong H^{k-1} ( M(\SA_{k-1}^r))^{G(r,1,k-1)} \otimes H^{1}
    (M(\SB_2))$.
  \item \label{it:xx} If $k=n-1$, then
    \begin{align*}
      H^{n-1}( M(\SA_{\tau_{n-1}} ))^{N(\tau_{n-1})} %
      & = H^{n-1}( M(\SA_{\tau_{n-1}} ))^{Z(\tau_{n-1})}\\
      & \cong H^{n-2} (M(\SA_{n-2}(r) ))^{G(r, p,n-2)}  \otimes H^{1}
        (M(\SB_2)).  
    \end{align*}
  \end{enumerate}
\end{lemma}

\begin{proof}
  The proof of \cref{it:lema01} is similar to the proof of \cref{it:lema02}
  and is omitted. We prove \cref{it:lema02} and \cref{it:xx}.

  Set $\tau=\tau_k$. In coordinates, $X_{\tau}$ is the set of all vectors in
  $\BBC^n$ of the form $[0\, a\, a\, v]^t$, where $0\in \BBC^{k-1}$, $a\in
  \BBC$, $v\in \BBC^{n-k-1}$, and the superscript $t$ denotes transpose.

  It follows from \cref{ssec:xz}\cref{eq:zlamfac} that
  \begin{itemize}
  \item $Z(\tau) = G(r,p,k-1) \boxtimes W_2$ is the group of block diagonal
    matrices $\left[ \begin{smallmatrix} g_{k-1} && \\  &w_2& \\
        &&I_{n-k-1} \end{smallmatrix} \right]$, where $g_{k-1}\in G(r,p,k-1)$,
    $w_2\in W_2$, and $I_{n-k-1}$ is the identity matrix, and
  \item $N(\tau)$ is the group of block diagonal matrices
    $\left[ \begin{smallmatrix} w_{k-1}d && \\ & w_2e & \\ 
        && w_{n-k-1}f \end{smallmatrix} \right]$, where $w_j\in W_j$, $d\in
    D_{r,k-1}$, $e$ is a scalar matrix in $D_{r,2}$, $f\in D_{r,n-k-1}$, and
    $\det d \cdot \det e \cdot \det f \in \mu_{r/p}$.
  \end{itemize}
  
  For any divisor $q$ of $r$, let
  \[
    Z^q=Z_{G(r,q,n)} (X_\tau)\quad \text{and}\quad N^q=N_{G(r,q,n)} (X_\tau) ,
  \]
  be the pointwise and setwise stabilizers of $X_\tau$ in $G(r,q,n)$,
  respectively.

  By \cref{ssec:xz}\cref{eq:anlamfac}, $\SA_{X_\tau} \cong \SA_{k-1}(r) \times
  \SB_2$, and taking $p=1$, it follows from the K\"unnneth Theorem that there
  is an isomorphism
  \begin{equation}
    \label{eq:3}
    H^{k-1} (M( \SA_{k-1}(r))) \otimes H^1(M(\SB_2)) \xrightarrow{\ \cong\ }
    H^k(M( \SA_{X_\tau}))    
  \end{equation}
  that intertwines the $G(r,1,k-1)\times W_2$-action on the left with the
  $N^1$-action on the right. Hence, there is a subspace, $U$, of $H^k(M(
  \SA_{X_\tau} ))$ such that $U\cong H^{k-1} (M( \SA_{k-1}(r)))$ and
  \[
    H^k(M(\SA_{X_\tau} )) = U (\BBC h_{k,k+1}),
  \]
  where the right hand side denotes the set of sums of products. If
  $g_{k-1}\in G(r,1,k-1)$, $g_{n-k-1}\in G(r,1,n-k-1)$, and $u\in U$, then
  \begin{equation}
    \label{eq:4mu}
    \left[ \begin{smallmatrix} g_{k-1} && \\ &I_2 & \\
        &&g_{n-k-1} \end{smallmatrix} \right] \cdot u\, h_{k,k+1}  =
    (g_{k-1} \cdot u)\, h_{k,k+1}.
  \end{equation}

  Suppose $k<n-1$. Then, given $g_{k-1} \in G(r,1,k-1)$, it is possible to
  find a diagonal matrix $d_{n-k-1}\in D_{r,n-k-1}$ so that
  $\left[ \begin{smallmatrix} g_{k-1} && \\ &I_2& \\
      && d_{n-k-1} \end{smallmatrix} \right] \in N^p=N(\tau)$. Hence, it
  follows from \cref{eq:3} and \cref{eq:4mu} that
  \[
    H^{k}( M(\SA_{\tau} ))^{N(\tau)} \cong H^{k-1} ( M(\SA_{k-1}(r)
    ))^{G(r,1,k-1)} \otimes H^{1} (M(\SB_2)) .
  \]
  This proves \cref{it:lema02}.
 
  Now suppose $k=n-1$. Then $\tau=(2)$ and in coordinates, $X_\tau$ is the set
  of all vectors in $\BBC^n$ of the form $[0\, a\, a]^t$, where $0\in
  \BBC^{n-2}$ and $a\in \BBC$, and
  \[
    H^{n-1}(M( \SA_{\tau})) ^{Z(\tau)} \cong H^{n-2} (M( \SA_{n-2}(r)))
    ^{G(r,p,n-2)} \otimes H^1(M(\SB_2)) .
  \]
  It is easy to check that $N(\tau)$ acts trivially on $\SA_{X_\tau}$ when
  $n=2$ and $n=3$, and so
  \[
    H^{n-1}( M( \SA_{\tau} ))^{N(\tau)} = H^{n-1}( M( \SA_{\tau} ))^{Z(\tau)}
  \]
  in these cases.

  In the rest of the proof we assume that $n\geq 4$. Define $\epsilon =1$ if
  $p$ or $n$ is odd and $\epsilon =2$ otherwise. With this notation there are
  containments
  \begin{equation}
    \label{eq:et}
    \vcenter{\vbox{\xymatrix{ Z^p \ar@{^{(}->}[r] \ar@{^{(}->}[d] &
          Z^\epsilon  \ar@{^{(}->}[d] \\
          N^p \ar@{^{(}->}[r] & N^\epsilon} }} \quad\text{and}\quad
    \vcenter{\vbox{ \xymatrix{ H^{n-1}(M(\SA_{\tau}))^{N^\epsilon}
          \ar@{^{(}->}[r] \ar@{^{(}->}[d] &
          H^{n-1}(M(\SA_{\tau}))^{Z^\epsilon} \ar@{^{(}->}[d] \\ 
          H^{n-1}(M(\SA_{\tau}))^{N^p} \ar@{^{(}->}[r] &
          H^{n-1}(M(\SA_{\tau}))^{Z^p} .} }} 
  \end{equation}

  Notice that $p$ and $n$ are both even if and only if $\epsilon$ and $n-2$
  are both even, so by induction $H^{n-1}(M(\SA_{\tau}))^{Z^\epsilon}$ and
  $H^{n-1}(M(\SA_{\tau}))^{Z^p}$ have the same dimension and hence must be
  equal. Moreover, $N^\epsilon = Z^\epsilon \cdot Z(G(r,1,n))$, and so it
  follows from \cref{lem:zfac} that $H^{n-1}(M(\SA_{\tau}))^{N^\epsilon} =
  H^{n-1} (M(\SA_{\tau}) )^{Z^\epsilon}$. This shows that all four spaces in
  the second diagram in \cref{eq:et} are equal. Therefore,
  \[
    H^{n-1}(M(\SA_{\tau}))^{N^p} = H^{n-1}(M(\SA_{\tau}))^{Z^p} \cong H^{n-2}
    (M(\SA_{n-2}(r) ))^{G(r,p,n-2)} \otimes H^{1} (M(\SB_2)) .
  \]
  This proves \cref{it:xx}.
\end{proof}

\subsection{}

Now consider the entries in \cref{tab:thm32} in the rows indexed by
$\SA=\SA_n(r)$.

Suppose $k=0$ and $\eta=\eta_0$. Then we have seen that
\[
  \dim H^{0} ( M(\SA _{\eta} )) ^{N(\eta)} = \dim H^{0} ( \BBC^n) =1.
\]
This justifies the entries in the column indexed by $k=0$ in \cref{tab:thm32}.

Suppose $1\leq k\leq n-2$, $\eta=\eta_k$, and $\tau =\tau_k$. Then it follows
from \cref{lem:a0} and induction that
\begin{align*}
  &\dim H^{k}(M(\SA_{\eta}))^{N(\eta)} =\dim H^{k} (M(\SA_k(r)) )^{G(r,1,
    k)} = 1 \\ \intertext{and} 
  &\dim H^{k}(M(\SA_{\tau} ))^{N(\tau)} =\dim H^{k-1} (M(\SA_{k-1}(r))
    )^{G(r,1,k-1)} = 1 .
\end{align*}
This justifies the entries in the column indexed by $1\leq k \leq n-2$ in
\cref{tab:thm32}.

Suppose $k=n-1$, $\eta=\eta_{n-1}$ and $\tau =\tau_{n-1}$. Then it follows
from \cref{lem:a0} and induction that
\begin{align*}
  &\dim H^{n-1}(M(\SA_{\eta}))^{N(\eta)} =\dim H^{n-1} (M(\SA_{n-1}(r))
    )^{G(r,1, n-1)} = 1 \\ \intertext{and}
  &\dim H^{n-1}(M(\SA_{\tau}))^{N(\tau)} =\dim H^{n-2}
    (M(\SA_{n-2}(r)))^{G(r,p,n-2)} =
  \begin{cases}
    2&\text{if $p$ and $n$ are even,} \\
    1&\text{if $p$ or $n$ is odd,}
  \end{cases}
\end{align*}
because
\[
  \dim H^{1}(M(\SA_2(r)_{(2)}))^{N((2))} =
  \begin{cases}
    2&\text{if $p$ is even,} \\
    1&\text{if $p$ is odd.}
  \end{cases}
\]
This justifies the entries in the columns indexed by $k =n-1$ in
\cref{tab:thm32}. 

Finally, suppose $k=n$. Then it follows from \cref{ssec:cpp}\cref{eq:cn} and
what has already been proved that
\[
  \dim H^{n}(M(\SA))^{G} =
  \begin{cases}
    2&\text{if $p$ and $n$ are even,} \\
    1&\text{if $p$ or $n$ is odd.}
  \end{cases}
\]
This justifies the entries in the columns indexed by $k =n$ in
\cref{tab:thm32} and completes the proof of \cref{thm:3}\cref{it:thm32} when
$\SA=\SA_n(r)$.


\section{A basis of $H^*(M(\SA))^G$}\label{sec:basis}

In this section we turn to the construction of a basis of $H^*(M(\SA))^G$ for
each reflection pair $(\SA,G)$. We may assume that $(\SA,G)$ is irreducible.
The strategy is to find bases of $H^{\rk T} (M(\SA_{X_T}))^{N_G(X_T)}$, for
$T\in \CT (\SA,G)^{\tdi}$, and then use the equality
\[
  H^k(M(\SA))^G =\sum_{T\in \CT(\SA,G)^{\tdi}_k}e_G\cdot H^k(M(\SA_{X_T}
  ))^{N_G(X_T)}
\]
in \cref{pro:bries}\cref{it:bries2}. Most of the time $H^{\rk T}
(M(\SA_{X_T}))^{N_T}$ is one-dimensional, so it is enough to identify a
non-zero element in $H^{\rk T} (M(\SA_{X_T}))^{N_T}$. The key technical result
we use is the acyclic complex in \cref{ssec:cont}.

\subsection{}
Suppose $(\SA,G)$ is a reflection pair with underlying vector space $V$. As
observed in \cref{ssec:tdi} that if $T=A_0$, then $\{T\}=
\CT(\SA,G)^{\tdi}_0$, $M(\SA_{X_T})=V$, $N_G(X_T)=G$, and $\{1= e_G\cdot 1\}$
is a basis of $H^0(M(\SA_{X_T}))^{N_G(X_T)}=H^0(M(\SA))^G$. Define
\[
  \cx^{\SA,G}_{T}=1 \in H^0(M(\SA_{X_T})) \quad\text{and}\quad
  B^{\SA,G}_{T}= \{ \cx^{\SA,G}_{T}\} 
\]
($\cx$ for Coxeter). Similarly, if $T=A_1$ and $H=X_T$, then $T\in
\CT(\SA,G)^{\tdi}_1$, $e_{N_G(H)}\cdot \{h\}$ is a basis of
$H^1(M(\SA_{X_T}))^{N_G(X_T)}$, and $e_Ge_{N_G(X_T)} \cdot h= e_G \cdot h$ is
a non-zero element in $H^1(M(\SA)^{G}$. Define
\[
  \cx^{\SA,G}_{T}=h \in H^1(M(\SA_{X_T})) \quad\text{and}\quad B^{\SA,G}_{T}=
  \{ \cx^{\SA,G}_{T}\} .
\]

For any $k$ it follows from \cref{pro:bries} that
\[
  H^k(M(\SA)) \cong \sum_{T\in \CT(\SA,G)^{\tdi}_k} \Big( \sum_{Y\in GX_T}
  H^{k}(M(\SA_Y)) \Big) ,
\]
that $\sum_{Y\in GX_T} H^{k}(M(\SA_Y))$ is a $\BBQ G$-submodule of
$H^{k}(M(\SA))$, isomorphic to the induced module $\Ind_{N_G(X_T)}^G
H^{k}(M(\SA_{X_T}))$, and that
\[
  \Big(\sum_{Y\in GX_T} H^{k}(M(\SA_Y)) \Big)^G= e_G \cdot H^k (M(\SA_{X_T}))
  = e_G\cdot H^{k}(M(\SA_{X_T}))^{N_G(X_T)}.
\]

Suppose for the moment that $\dim V=2$ and that $\{H_1, \dots, H_a\}$ is a set
of $G$-orbit representatives in $\SA$. If $a>1$, then $\{G\}=
\CT(\SA,G)^{\tdi}_2$. Define
\[
  \cx^{\SA,G}_{G}=h_1h_2\quad\text{and}\quad B^{\SA,G}_{G}= \{h_1h_2,\dots,
  h_1h_a \} .
\]
The next corollary follows from \cref{pro:b2}.

\begin{corollary}\label{cor:prot}
  Suppose $(\SA,G)$ is an irreducible reflection pair with $\rk \SA=2$ and set
  $\CT=\CT(\SA,G)$.
  \begin{enumerate}
  \item  For $T\in \CT^{\tdi}$, the set $e_{N_G(X_T)} \cdot
    B^{\SA,G}_T$ is a basis of $H^{\rk T}(M(\SA_{X_T} ))^{N_G(X_T)}$.
  \item For $T\in \CT^{\tdi}$, the set $e_{G} \cdot B^{\SA,G}_T$ is a basis of
    $\Big(\sum_{Y\in GX_T} H^{\rk X_T}(M(\SA_Y)) \Big)^G$.
  \item For $k\geq 0$, the disjoint union \,$\coprod_{T\in \CT^{\tdi}_k}
    e_G\cdot B^{\SA,G}_T$ is a basis of $H^{k}(M(\SA ))^{G}$.
  \end{enumerate}
\end{corollary}

The main result in this section is the analog of the statement in
\cref{cor:prot}, but without the assumption that $\rk \SA =2$. In order for
this to make sense, we need to define the sets $B^{\SA, G}_{T}$ when $(\SA,G)$
has rank greater than two and $T\in \CT(\SA,G)^{\tdi}$. It turns out that when
$\dim H^{\rk T}(M(\SA_{X_T} ))^{N_G(X_T)}=1$ we can take $B^{\SA,G}_{T}= \{
\cx^{\SA,G}_{T}\}$, where $\cx^{\SA,G}_{T} \in H^{\rk T}(M(\SA_{X_T} ))$ is a
``Coxeter-like'' cohomology class.

\subsection{}\label{ssec:cx}
Until \cref{ssec:thm3ve}, $(\SA,G)$ is an irreducible reflection pair with
rank at least three that is either primitive or imprimitive.

\begin{enumerate}
\item First suppose $(\SA,G)$ is primitive, so $\SA=\SA(G)$ and $G$ is an
  exceptional reflection group. \label{it:cxprim}

  Let $T\in \CT(\SA,G)^{\tdi}\setminus \{A_0,A_1\}$. Then $\dim H^{\rk
    T}(M(\SA_{X_T}))^{N_G(X_T)}=1$, by \cref{thm:3}\cref{it:thm33}, and so
  \[
    H^{k}(M(\SA_{X_T}))^{N_G(X_T)} \cong H^k(M(\SA(Z_G(X_T))))^{Z_G(X_T)},
  \]
  where $k=\rk T$. There is a BMR diagram for $G$, say $\CD$, such that a BMR
  diagram for $T$ appears as an admissible subdiagram of $\CD$ (see
  \cite[\S1.B] {brouemallerouquier:complex}). Let $s_1$, \dots, $s_k$ be
  generators of $G$ indexed by the nodes in the admissible subdiagram for
  $T$. For $1\leq i\leq k$ define hyperplanes $H_i=\Fix(s_i)\in \SA$ with
  Orlik-Solomon generators $h_i \in H^1(M(\SA))$. Then $Z_G(X_T)= \langle s_1,
  \dots, s_k\rangle$ by \cite[1.7(1)] {brouemallerouquier:complex}.

  If $T$ is well-generated, then $X_T= \Fix(Z_G(X_T))= H_1\cap \dotsm \cap
  H_k$. Define
  \[
    \cx^{\SA,G}_T= h_1\dotsm h_k \in H^k(M(\SA_{X_T})).
  \]
  If $T$ is not well-generated, then $G=G_{31}$ and $T\in \{G(4,2,3),
  G_{31}\}$. In this case, with the notation in \cite[Tab.~3]
  {brouemallerouquier:complex}, define
  \[
    \cx^{\SA,G}_{T}=
    \begin{cases}
      h_v h_t h_u\in H^{3} (M(\SA_{X_T}))&\text{if $T=G(4,2,3)$,} \\
      h_v h_t h_u h_w\in H^{4} (M(\SA))&\text{if $T=G_{31}$.}
    \end{cases}
  \]
  
\item \label{it:cximprim} Now suppose $(\SA,G)$ is imprimitive and $T\in
  \CT(\SA, G)^{\tdi} \setminus \{A_0, A_1\}$. Then $G=G(r,p,n)$ and there are
  two cases.

  \noindent $\bm{\SA= \SA_n(r)}$: By \cref{thm:3}\cref{it:thm32},
  \begin{align*}
    \CT(\SA, G)^{\tdi}%
    & = \{A_0, A_1\} \amalg \{\, G(r,p,k)\mid 1\leq k\leq n\,\} \\
    &\qquad\amalg \{\, G(r,p,k-1)A_1 \mid 2\leq  k\leq n-1\,\} . 
  \end{align*}
    
  With the notation in \cite[3.A] {brouemallerouquier:complex} and
  \cite[Tab.~1] {brouemallerouquier:complex}, the group $G(r,1,n)$ (denoted by
  $G(d,1,r)$ in \cite {brouemallerouquier:complex}) has generators $s$, $t_2$,
  \dots, $t_n$. These define hyperplanes $\Fix(s)$, $\Fix(t_2)$, \dots,
  $\Fix(t_n)$ in $\SA$ with corresponding Orlik-Solomon generators
  $h_\varphi$, $h_2$, \dots, $h_n$ in $H^1(M(\SA))$. Using the notation in
  \cref{ssec:rc} and \cref{ssec:grpn}, we may take $s=\varphi$ and
  $\Fix(t_j)=H_{j-1,j}(1)$ for $2\leq j\leq n$.

  If $\dim H^{\rk T}(M(\SA_{X_T} ))^{N_G(X_T)}=1$, define
  \[
    \hspace{6em} \cx^{\SA,G}_T=
    \begin{cases}
      h_\varphi h_{2}h_{3} \dotsm h_{k},&\text{if $T=G(r,p,k)$ and $1\leq
        k\leq n-1$,} \\
      h_\varphi h_{2}h_{3} \dotsm h_{k-1}h_{k+1}, &\text{if $T=G(r,p,k-1)
        A_1$ and  $2\leq k\leq n-1$,} \\
      h_\varphi h_{2}h_{3} \dotsm h_{n}, &\text{if $T=G$ and either $p$ or $n$
        is odd.}
    \end{cases}
  \]

  If $\dim H^{\rk T}(M(\SA_{X_T} ))^{N_G(X_T)}>1$, then $T\in \{G(r,p,n-2)
  A_1, G\}$ and $p$ and $n$ are even. To simplify the notation, set
  $\Ttilde=G(r,p,n-2) A_1$. Because $p$ is even, $G(r,p,n)$ is not
  well-generated unless $p=r$, which is not under consideration because
  $\SA=\SA_n(r)= \SA(G)$. Define a hyperplane $H_2'=\Fix(st_2s\inverse)\in
  \SA$, corresponding to the BMR generator $t_2'$ in \cite[Tab.~2]
  {brouemallerouquier:complex}, and then define
  \[
    \begin{aligned} \cx^{\SA,G}_{\Ttilde}&=h_\varphi h_{2}h_{3} \dotsm
      h_{n-2}h_n,%
      &\cx^{\SA,G}_{\Ttilde,1}&=h_\varphi h_2' h_{3} \dotsm h_{n-2}h_n,\\
      \cx^{\SA,G}_{G}&=h_\varphi h_{2}h_{3} \dotsm h_{n-1}h_n,%
      &\cx^{\SA,G}_{G,1}&=h_\varphi h_2' h_{3} \dotsm h_{n-1}h_n,
    \end{aligned}
  \]
  and
  \[
    B^{\SA,G}_{T}= \{ \cx^{\SA,G}_T, \cx^{\SA,G}_{T,1}\} \subseteq H^{\rk
      T}(M(\SA_{X_T}))
  \]
  for $T\in \{\Ttilde,G\}$.
  \smallskip
  
  \noindent $\bm{\SA= \SA_n^0(r)}$: By \cref{thm:3}\cref{it:thm32}, because
  $T\notin \{A_0,A_1\}$ we have that $p$ and $n$ are even, $T\in
  \{G(r,p,n-2)A_1, G\}$, and $\dim H^{\rk T} (M(\SA_{X_T} ))^{N_G(X_T)}
  =1$. With the notation just introduced, notice that the hyperplanes $H_2$,
  $H_2'$, $H_3$, \dots, $H_n$ lie in $\SA_n^0(r)$. Define
  \[
    \cx^{\SA,G}_T= \begin{cases}
      h_2h_2'h_{3} \dotsm h_{n-2}h_n,&\text{if $T=G(r,p,n-2) A_1$,} \\
      h_2h_2'h_{3} \dotsm h_{n},&\text{if $T=G$.}
    \end{cases}
  \]
\end{enumerate}

Finally, for $(\SA,G)$ primitive or imprimitive, define
\[
  B^{\SA,G}_{T}= \{ \cx^{\SA,G}_{T}\} \quad\text{for all $T$ with $\dim H^{\rk
      T} (M(\SA_{X_T} ))^{N_G(X_T)}=1$.}
\]
We can now state the main result in this section.

\begin{theorem}\label{thm:4}
  Suppose $(\SA,G)$ is an irreducible reflection pair and set
  $\CT=\CT(\SA,G)$.
  \begin{enumerate}
  \item \label{it:22t} For $T\in \CT^{\tdi}$, the set $e_{N_G(X_T)} \cdot
    B^{\SA,G}_T$ is a basis of $H^{\rk T}(M(\SA_{X_T} ))^{N_G(X_T)}$.
  \item \label{it:2at} For $T\in \CT^{\tdi}$, the set $e_{G} \cdot
    B^{\SA,G}_T$ is a basis of $\Big(\sum_{Y\in GX_T} H^{\rk X_T}(M(\SA_Y))
    \Big)^G$.
  \item \label{it:2bt} For $k\geq 0$, the disjoint union \,$\coprod_{T\in
      \CT^{\tdi}_k} e_G\cdot B^{\SA,G}_T$ is a basis of $H^{k}(M(\SA ))^{G}$.
  \end{enumerate}
\end{theorem}

The next corollary contains as a special case a conjecture of Felder and
Veselov \cite{felderveselov:coxeter} for Coxeter groups.

\begin{corollary}\label{cor:4}
  If $G$ is a well-generated complex reflection group, then
  \[
    \{\, e_G\cdot \cx^{\SA(G),G}_T \mid T\in \CT(\SA(G),G)^{\tdi}\,\}
  \]
  is a basis of $H^*(M(\SA))^G$.
\end{corollary}

\subsection{Proof of \cref{thm:4}}

The assertions in \cref{thm:4}\cref{it:2at} and \cref{thm:4}\cref{it:2bt}
follow from \cref{thm:4}\cref{it:22t} and \cref{pro:bries}.

As with \cref{thm:3}, the proof of \cref{thm:4}\cref{it:22t} is by recursion
for primitive reflection pairs, by induction on $n$ for imprimitive reflection
pairs, and case-by-case for very exceptional reflection pairs. To illustrate
the method, we provide details for the reflection pairs $(\SG_{34}, G_{34})$
and $(\SA_4(r),G(r,p,4))$ when $p$ is even. The extension to other cases is
straightforward and is omitted.

\subsection{}\label{ssec:bg34}
Consider the reflection pair $(\SG, G)=(\SG_{34}, G_{34})$. By
\cref{thm:3}\cref{it:thm33}, $\CT(\SG, G)^{\tdi}= \{A_0, A_1,G_{33}, G\}$. We
have seen that $e_{N_G(X_T)}\cdot \cx^{\SG,G}_T \ne0$ for $T\in \{A_0,
A_1\}$. In the rest of this subsection, set $T=G_{33}$.

Using the second BMR diagram for $G_{34}$ (\cite[Tab.~4]
{brouemallerouquier:complex}) we have
\[
  \cx^{\SG,G}_{G}= h_\varphi h_th_u h_vh_{w'}h_x \quad\text{and}\quad
  \cx^{\SG,G}_{T}= \cx^{\SG_{33},G_{33}}_{G_{33}}= h_\varphi h_th_u h_vh_{w'} .
\]
Assuming that we have proved the theorem for all $T\in \CT(\SG, G)\setminus
\{G\}$, we know that $e_{G_{33}} \cdot \cx^{\SG_{33}, G_{33}}_{G_{33}} \ne0$
in $H^5(M(\SG_{33}))^{G_{33}}$. Using \cref{thm:3}\cref{it:thm33} again, there
is an isomorphism
\[
  H^5(M(\SG_{33}))^{G_{33}} \cong H^5(M(\SG_{X_T}))^{Z_G(X_T)} =
  H^5(M(\SG_{X_T}))^{N_G(X_T)},
\]
that maps $e_{G_{33}} \cdot \cx^{\SG_{33},
  G_{33}}_{G_{33}} \ne0$ to $e_{N_G(X_T)}\cdot \cx^{\SG,G}_T$. Hence
$e_{N_G(X_T)}\cdot \cx^{\SG,G}_T \ne0$.

It remains to show that $e_{G}\cdot \cx^{\SG,G}_G \in H^6(M(\SG))^G$ is not
equal to zero. Applying the map $\partial$ in \cref{ssec:cont} to $e_G \cdot
\cx^{\SA,G}_{G}$ we get
\begin{multline}\label{eq:p}
  \partial( e_G\cdot h_\varphi h_th_u h_vh_{w'}h_x)= e_G\cdot h_th_u
  h_vh_{w'}h_x -e_G\cdot h_\varphi h_u h_vh_{w'}h_x \\
  +e_G\cdot h_\varphi h_t h_vh_{w'}h_x -e_G\cdot h_\varphi h_th_u h_{w'}h_x
  +e_G\cdot h_\varphi h_th_u h_vh_x -e_G\cdot h_\varphi h_th_u h_vh_{w'} .
\end{multline}

Consider the first summand, $e_G\cdot h_th_u h_vh_{w'}h_x$. Set $X= H_t\cap
H_u \cap H_v \cap H_{w'}\cap H_x$. It follows from Steinberg's Theorem that
the pointwise stabilizer of $X$ is $Z_G(X)= \langle t, u, v, w',x \rangle$,
which has reflection type $D_5$. Thus, for the reflection type $\Ttilde=D_5$
in $\CT(\SG,G)$, we may take $X_{\Ttilde}=X$. Then $e_{N_G(X_{\Ttilde})}\cdot
\cx^{\SG,G}_{\Ttilde} = e_{Z_G(X)}\cdot h_th_u h_vh_{w'}h_x \in
H^5(M(\SG_{X})) ^{Z_G(X)}$. But $H^5(M(\SG_{X}))^{Z_G(X)} \cong
H^5(M(\SA(D_5)))^{D_5}=0$, where the last equality follows from
\cref{thm:3}\cref{it:thm32}. Thus $e_{Z_G(X)}\cdot h_th_u h_vh_{w'}h_x =0$,
and hence $e_G\cdot h_th_u h_vh_{w'}h_x=0$.

Similar arguments show that all the summands in \cref{eq:p} are equal to zero
except the last. Arguing as in the preceding paragraph, with $X= H_\varphi
\cap H_t\cap H_u \cap H_v \cap H_{w'}$, we have $Z_G(X)= \langle s,t, u, v, w'
\rangle$, which has reflection type $G_{33}$. Therefore, $e_G\cdot h_\varphi
h_th_u h_vh_{w'}= e_G\cdot \cx^{\SA,G}_{G_{33}}$.

Putting the pieces together, $\partial(e_G \cdot \cx^{\SA,G}_{G}) = -e_G\cdot
\cx^{\SA,G}_{G_{33}} \ne0$. Since $\partial$ is injective, we conclude that
$e_{G}\cdot \cx^{\SA,G}_G \ne0$, as desired.

\subsection{}
Consider the reflection pair $(\SA, G)=(\SA_4(r), G(r,p,4))$, where $p$ is
even. By induction, we may assume that the conclusions of the theorem hold for
$n=2$ and $n=3$.

It follows from \cref{thm:3}\cref{it:thm32} that
\[
  \CT(\SA,G)^{\tdi}= \{A_0, A_1, G(r,p,1), G(r,p,1)A_1, G(r,p,2), G(r,p,2)A_1,
  G(r,p,3), G\} ,
\]
and that $\dim H^{\rk T}(M(\SA_{X_T} ))^{N_G(X_T)}=1$ unless $T\in \{
G(r,p,2)A_1, G\}$. For the rest of this subsection, set $\Ttilde=
G(r,p,2)A_1$. The elements $\cx^{\SA,G}_T$ and the sets $B^{\SA, G}_{\Ttilde}$
and $B^{\SA,G}_{G}$ are given in \cref{tab:thm4}.

\begin{table}[htb]
  \caption{$\cx^{\SA,G}_T$ and $B^{\SA,G}_{T}$ for $(\SA_4(r), G(r,p,4))$, $p$
    even}
  \begin{minipage}{\linewidth}
    \centering   \renewcommand{\arraystretch}{1.3}
    \begin{tabular} {>{$}l<{$} @{\hspace {1em}}>{$}c<{$} >{$}c<{$} >{$}c<{$}
        >{$}c<{$} >{$}c<{$} >{$}c<{$} >{$}c<{$} >{$}c<{$}}
      \toprule \addlinespace
      \deg& k=0 %
      & \multicolumn{2}{c}{$k=1$}%
      & \multicolumn{2}{c}{$k=2$}
      & \multicolumn{2}{c}{$k=3$} & k=4 \\  
      \cmidrule(r{1.5em}){1-1} \cmidrule{2-2} \cmidrule(lr){3-4}
      \cmidrule(lr){5-6} \cmidrule(lr){7-8} \cmidrule(lr){9-9}
      T&A_0&A_1&G_{r,p,1}%
      & G_{r,p,1} A_1  & G_{r,p,2}%
           &G_{r,p,2} A_1 & G_{r,p,3} &G_{r,p,4}\\ 
%
      \cx^{\SA,G}_T& 1 & h_2%
      &h_\varphi& h_\varphi h_3 & h_\varphi h_2& .%
               &h_\varphi h_2h_3 & .\\
      \addlinespace \midrule \addlinespace
      B^{\SA,G}_{\Ttilde}%
          & \multicolumn{6}{l}{$\{ \cx^{\SA,G}_{\Ttilde},
            \cx^{\SA,G}_{\Ttilde,1}\}=\{\ h_\varphi h_{2}h_{4},\ 
            h_\varphi h_2'h_{4}\ \}$ }\\
      B^{\SA, G}_{G}& \multicolumn{7}{l}{ $\{ \cx^{\SA,G}_G,
                      \cx^{\SA,G}_{G,1}\}=\{\ h_\varphi h_{2}h_3h_{4},\ 
                      h_\varphi h_2'h_3h_{4}\ \}$ }\\
      \addlinespace \bottomrule \addlinespace
    \end{tabular}
  \end{minipage}
  \label{tab:thm4}
\end{table}

We need to show that
\begin{itemize}
\item if $T\in \CT(\SA,G)^{\tdi}$ and $\dim H^{\rk T}(M(\SA_{X_T}
  ))^{N_G(X_T)} =1$, then $e_{N_G(X_T)}\cdot \cx^{\SA,G}_T \ne0$, and
\item if $T\in \{ \Ttilde, G\}$, then $e_{N_G(X_T)} \cdot B^{\SA,G}_{T}$ is a
  basis of $H^{\rk T}(M(\SA_{X_T} ))^{N_G(X_T)}$.
\end{itemize}
As in the preceding subsection, the results are known for $T\in \{A_0,A_1,
G(r,p,1)\}$ (recall from \cref{ssec:et} that $G$ has two orbits on $\SA$,
indexed by $A_1$ and $G(r,p,1)$).

Suppose that $T=G(r,p,k)$ and $k=2,3$. It was shown in
\cref{lem:a0}\cref{it:lema01} that
\begin{equation}
  \label{eq:8}
  H^{k}( M(\SA_{X_T} ))^{N_G(X_T)} \cong H^{k} ( M(\SA_k(r)))^{G(r,1,k)}
\end{equation}
by an isomorphism that is induced by the projection $\BBC^4 \to \BBC^4/X_T$
and that intertwines the $N_G(X_T)$-action on the left with the
$G(r,1,k)$-action on the right. By induction, the space $H^{k} (
M(\SA_k(r)))^{G(r,1,k)}$ is one-dimensional with basis $\{e_{G(r,1,k)} \cdot
h_\varphi\dotsm h_k\}$. It is straightforward to check that the isomorphism in
\cref{eq:8} maps $h_\varphi\dotsm h_k$ in $H^{k}( M(\SA_{X_T} ))$ to
$h_\varphi\dotsm h_k$ in $H^{k} ( M(\SA_k(r)))$. Therefore, $e_{N_G(X_T)}
\cdot h_\varphi\dotsm h_k$ maps to $e_{G(r,1,k)} \cdot h_\varphi\dotsm h_k$,
and hence $e_{N_G(X_T)} \cdot h_\varphi\dotsm h_k \ne 0$. This shows that
$e_G\cdot \cx^{\SA,G}_T \ne0$ for $T\in \{G(r,p,2), G(r,p,3)\}$. In particular
\begin{equation}
  \label{eq:11}
  \text{$e_{N_G(X_{G(r,p,3)})} \cdot \{h_\varphi h_2h_3\}$ is a basis of
    $H^{3}( M(\SA_{X_{G(r,p,3)}} ))^{N_G(X_{G(r,p,3)})}$. }
\end{equation}

Similarly, if $T=G(r,p,1) A_1$, then it follows from the isomorphism
\[
  H^{2}( M(\SA_{X_T} ))^{N_G(X_T)} \cong H^{1} ( M(\SA_{1}^r) )^{G(r,1,1)}
  \otimes H^{1} (M(\SB_2))
\]
in \cref{lem:a0}\cref{it:lema02} that $e_{N_G(X_T)}\cdot \cx^{\SA,G}_T \ne0$.

If $T=G(r,p,2) A_1$, then by \cref{lem:a0}\cref{it:xx},
\[
  H^{3}( M(\SA_{X_T} ))^{N_G(X_T)} \cong H^{2} (M(\SA_{2}(r) ))^{G(r, p,2)}
  \otimes H^{1} (M(\SB_2)) .
\]
By induction, $e_{G(r,p,2)} \cdot \{ h_\varphi h_2, h_\varphi h_2'\}$ is a
basis of $H^{2} (M(\SA_{2}(r) ))^{G(r, p,2)}$, and then it follows that
\begin{equation}
  \label{eq:12}
  \text{ $e_{N_G(X_{\Ttilde})} \cdot B^{\SA,G}_{\Ttilde} =
    e_{N_G(X_{\Ttilde})} \cdot \{h_\varphi h_2h_4, h_\varphi h_2'h_4\}$ is a
    basis of $H^{3}(M(\SA_{X_{\Ttilde}} ))^{N_G(X_{\Ttilde})}$. }
\end{equation}

Notice that by \cref{eq:11}, \cref{eq:12}, and \cref{thm:3}\cref{it:thm32},
\begin{equation}
  \label{eq:13}
  \text{$\{\, e_G\cdot h_\varphi h_2h_3,\, e_G\cdot h_\varphi h_2h_4,\,
    e_G\cdot h_\varphi h_2'h_4\,\}$ is a basis of $H^3(M(\SA))^G$.}
\end{equation}

Finally, consider $e_G\cdot B^{\SA,G}_{G}= e_G\cdot \{h_\varphi h_2h_3h_4,
h_\varphi h_2'h_3h_4\}$. One easily checks that $e_G\cdot h_\varphi h_2'h_3$
is a scalar multiple of $e_G\cdot h_\varphi h_2h_3$, say $e_G\cdot h_\varphi
h_2'h_3= \xi e_G\cdot h_\varphi h_2h_3$. Applying the map $\partial$ in
\cref{ssec:cont} and arguing as in \cref{ssec:bg34} we have
\begin{align*}
  \partial(e_G\cdot h_\varphi h_2h_3h_4)%
  &= e_G\cdot h_\varphi h_2h_4 - e_G\cdot h_\varphi h_2h_3 \\ 
  \intertext{and}
  \partial(e_G\cdot h_\varphi h_2'h_3h_4)%
  &= e_G\cdot h_\varphi h_2'h_4- e_G\cdot
    h_\varphi h_2'h_3\\
  &= e_G\cdot h_\varphi h_2'h_4- \xi e_G\cdot h_\varphi h_2h_3.
\end{align*}
Therefore, it follows from \cref{eq:13} that $\partial \big(e_G\cdot
B^{\SA,G}_{G}\big)$ is linearly independent, and so $e_G\cdot B^{\SA,G}_{G}$
is linearly independent and hence a basis of $H^{\rk G}(M(\SA))^{G}$. This
completes the proof of \cref{thm:4}.

\subsection{Very exceptional reflection pairs}\label{ssec:thm3ve}

In this subsection we prove \cref{thm:3}\cref{it:thm33} for the very
exceptional reflection pairs by using \cref{thm:4} to construct an explicit
basis of $H^*(M(\SA))^G$ for each such pair.

First consider the pair $(\SA_3^0(3), \langle G(3,3,3)\sigma \rangle)$. Set
$\SA=\SA_3^0(3)$ and $G= G(3,3,3)$. Fix a hyperplane, $H$, in $\SA_3^0(3)$.
By \cref{thm:3}\cref{it:thm32}, $\{1, e_G\cdot h\}$ is a basis of
$H^*(M(\SA))^G$. Because $G$ acts transitively on $\SA$, so does $\langle
G\sigma \rangle$. It follows that $H^*(M(\SA))^G = H^*(M(\SA))^{\langle
  G\sigma \rangle}$. This justifies the entries in the row for $(\SA_3^0(3),
\langle G(3,3,3)\sigma \rangle)$ in \cref{tab:thm33}.

Next consider the pair $(\SF_4, \langle F_4 \gamma\rangle)$. With the notation
in the BMR diagram for $F_4=G_{28}$, it follows from \cref{thm:4} that
\[
  e_{F_4}\cdot \{1, h_s, h_v, h_sh_v, h_th_u, h_sh_th_u, h_th_uh_v,
  h_sh_th_uh_v\}
\]
is a basis of $H^*(M(\SF_4))^{F_4}$. Also, the linear transformation $\gamma$
acts on $\{ h_s, h_t ,h_u, h_v\}$ by interchanging $h_s$ and $h_v$ and
interchanging $h_t$ and $h_u$. Thus $\langle F_4\gamma\rangle$ acts
transitively on $\SF_4$,
\[
  \gamma\cdot h_sh_v= -h_sh_v, \quad, \gamma\cdot h_th_u= -h_th_u,
  \quad\text{and}\quad \gamma\cdot h_sh_th_u= -h_th_uh_v ,
\]
and it follows easily that
\[
  e_{\langle F_4\gamma \rangle}\cdot \{1, h_s, h_sh_th_u- h_th_uh_v,
  h_sh_th_uh_v\}
\]
is a basis of $H^*(M(\SF_4))^{\langle F_4\sigma \rangle}$. This justifies the
entries in the row for $(\SF_4, \langle F_4\gamma\rangle)$ in
\cref{tab:thm33}.

Finally, for the pair $(\SA_4^0(2), \langle D_4\tau \rangle)$, it is easy to
see that $\tau$ acts trivially on the basis of $H^*(M(\SA_4^0(2)))^{D_4}$ in
\cref{thm:4} and so $H^*(M(\SA_4^0(2)))^{\langle D_4\tau \rangle}=
H^*(M(\SA_4^0(2) ))^{D_4}$.  This justifies the entries in the row for
$(\SA_4^0(2), \langle D_4\tau \rangle)$ in \cref{tab:thm33}.

\section{Application: Computing Lehrer's relative equivariant Poincar\'e
  polynomials}\label{sec:leh}

Lehrer \cite{lehrer:rationalpoints} considers the relative situation in which
$G\subseteq \Gtilde\subseteq \GL(V)$ are complex reflection groups such that
$G$ is irreducible and normal in $\Gtilde$, and $\SA=\SA(G)$ or
$\SA(\Gtilde)$. A fundamental result in \cite{lehrer:rationalpoints} is the
computation of the character of the representation of
$\Gtilde/G$ on $H^*(M(\SA))^G$. The bases of $H^*(M(\SA))^G$ described in
\cref{thm:4} can easily be used to give an alternate derivation of these
characters. Here we focus on the cases when
\begin{itemize}
\item $\SA=\SF_4$, $G=F_4$, and $\Gtilde = \langle F_4\gamma \rangle$, and
\item $\SA=\SA_n(r)$, $G=G(r,p,n)$, $\Gtilde = G(r,1,n)$, and $n\geq 3$.
\end{itemize}
The other cases are left to the reader. To simplify the notation a bit, for
$X\in L(\SA)$ set
\[
  \Ind_X^G= \sum_{Y\in GX} H^{\rk X}(M(\SA_Y)).
\]

\subsection{}
The computation for $\SA=\SF_4$, $G=F_4$, and $\Gtilde = \langle F_4\gamma
\rangle$ follows from the computation in \cref{ssec:thm3ve}. We want to
compute the graded character of $\langle \gamma\rangle$ on
$H^*(M(\SF_4))^{F_4}$.

Let $\epsilon_{\langle \gamma\rangle}$ denote the non-trivial character of
$\langle \gamma\rangle$. The following statements are immediate consequences
of the computations in \cref{ssec:thm3ve}:
\begin{itemize}
\item $H^0(M(\SF_4))^{F_4}= (\Ind_{V}^{F_4})^{F_4}$, a basis is $e_{F_4}\cdot
  \{1\}$, and $\langle \gamma\rangle$ acts trivially.
\item $H^1(M(\SF_4))^{F_4}= (\Ind_{H_s}^{F_4})^{F_4}+
  \big(\Ind_{H_t}^{F_4}\big)^{F_4}$, a basis is $e_{F_4}\cdot \{h_s+h_v,
  h_s-h_v\}$, and $\langle \gamma\rangle$ acts as $1_{\langle \gamma\rangle}
  +\epsilon_{\langle \gamma\rangle}$.
\item $H^2(M(\SF_4))^{F_4}= \big(\Ind_{H_s\cap H_v}^{F_4}\big)^{F_4}+
  \big(\Ind_{H_t\cap H_u}^{F_4}\big)^{F_4}$, a basis is $e_{F_4}\cdot
  \{h_sh_v, h_th_u\}$, and $\langle \gamma\rangle$ acts as $\epsilon_{\langle
    \gamma\rangle}+\epsilon_{\langle \gamma\rangle}$.
\item $H^3(M(\SF_4))^{F_4}= \big(\Ind_{H_s\cap H_t\cap H_u}^{F_4}\big)^{F_4}+
  \big(\Ind_{H_t\cap H_u\cap H_v}^{F_4}\big)^{F_4}$, a basis is
  \[
    e_{F_4}\cdot \{h_sh_th_u+ h_th_uh_v,\, h_sh_th_u -h_th_uh_v\, \},
  \]
  and $\langle \gamma\rangle$ acts as $1_{\langle \gamma\rangle}
  +\epsilon_{\langle \gamma\rangle}$.
\item $H^4(M(\SF_4))^{F_4}= \big(\Ind_{0}^{F_4} \big)^{F_4}$ is
  one-dimensional, a basis is $e_{F_4}\cdot \{ h_sh_th_uh_v \}$, and $\langle
  \gamma\rangle$ acts as $1_{\langle \gamma\rangle}$.
\end{itemize}
With the notation in \cite{lehrer:rationalpoints}, the equivariant Poincar\'e
polynomial is
\[
  P^{\langle \gamma \rangle}(X_{F_4},t)= (1+t+t^3+t^4) 1_{\langle \gamma
    \rangle} + (t+2t^2+t^3) \epsilon_{\langle \gamma\rangle} .
\]

\subsection{}
In the rest of this section consider the case when $\SA=\SA_n(r)$,
$G=G(r,p,n)$, $\Gtilde = G(r,1,n)$, and $n\geq 3$.

In general $\CT(\SA,G) \ne \CT(\SA, \Gtilde)$, but it follows from
\cref{thm:3} (see \cref{tab:thm32}) that we may canonically identify
$\CT(\SA,G)^{\tdi}$ and $\CT(\SA, \Gtilde) ^{\tdi}$ by replacing $p$ by
$1$. Moreover, we can choose a set of common orbit representatives.  For
example, if $\lambda = (1^{n-k})$, then $T=G(r,p,k)\in \CT(\SA, G)^{\tdi}$
corresponds to $\Ttilde = G(r,1,k)\in \CT(\SA, \Gtilde)^{\tdi}$, and we can
take $X_T=X_{\Ttilde}= X_\lambda$. Set
\[
  \CT^{\tdi}=\CT(\SA,G)^{\tdi}= \CT(\SA, \Gtilde)^{\tdi}.
\]

It is straightforward to check that if $T\in \CT^{\tdi}$, then $GX_T= \Gtilde
X_T$, whence
\[
  \Ind_{X_T}^\Gtilde = \sum_{Y\in \Gtilde X_T} H^{\rk T}(M(\SA_Y)) =\sum_{Y\in
    GX_T} H^{\rk T}(M(\SA_Y)) =\Ind_{X_T}^G.
\]
In particular, $\Ind_{X_T}^G$ is a $\BBQ \Gtilde$-module and
\begin{equation}\label{eq:odim}
  e_{\Gtilde} \cdot H^{\rk T}(M(\SA_{X_T})) = \big(\Ind_T^G\big)^{\Gtilde}
  \subseteq  \big(\Ind_T^{G}\big)^G= e_G \cdot H^{\rk T}(M(\SA_{X_T})) .
\end{equation}

If $r$ is even, let $\epsilon_{\Gtilde}$ denote the linear character of
$\Gtilde$ with order two that contains $W_n$ in its kernel. The next theorem
is a slight refinement of a result due to Lehrer
\cite[Thm~6.1]{lehrer:rationalpoints}.

\begin{theorem}
  Suppose $G=G(r,p,n)$, $\Gtilde= G(r,1,n)$, $\SA=\SA_n(r)$, and $T\in
  \CT^{\tdi}$. Then the character of the $\BBQ \Gtilde$-module
  $(\Ind_{X_T}^G)^G$ is equal to
  \[
    \begin{cases}
      1_{\Gtilde}+\epsilon_{\Gtilde}&\text{if $p$ and $n$ are even and $T\in
        \{G(r,p,n-2)A_1,G\}$,} \\
      1_{\Gtilde}&\text{otherwise.}
    \end{cases}
  \]
\end{theorem}

\begin{proof}
  Clearly $H^{\rk T}(M(\SA_{X_T}) )^{N_{\Gtilde}(X_T)} \subseteq H^{\rk
    T}(M(\SA_{X_T}) )^{N_{G}(X_T)}$. There are two cases, depending on whether
  or not this containment is an equality.
  
  Suppose first that $H^{\rk T}(M(\SA_{X_T}))^{N_{\Gtilde}(X_T)} = H^{\rk
    T}(M(\SA_{X_T}))^{N_{G}(X_T)}$. Then it follows from
  \cref{thm:3}\cref{it:thm32} that $e_{\Gtilde} \cdot H^{\rk T}(M(\SA_{X_T}))$
  and $e_G \cdot H^{\rk T}(M(\SA_{X_T}))$ are both one-dimensional, so
  equality holds in \cref{eq:odim}. Therefore $\Gtilde$ acts trivially on
  $\Ind_{X_T}^G$ and so the character of the $\BBQ \Gtilde$-module
  $(\Ind_{X_T}^G)^G$ is equal to $1_{\Gtilde}$, as claimed.

  Now suppose that $\dim H^{\rk T}(M(\SA_{X_T}))^{N_{\Gtilde}(X_T)} \ne H^{\rk
    T}(M(\SA_{X_T}))^{N_{G}(X_T)}$. Then by \cref{thm:3}\cref{it:thm32}, $p$
  and $n$ are both even, $T\in \{G(r,p,n-2)A_1, G\}$, $\dim H^{\rk
    T}(M(\SA_{X_T}))^{N_{\Gtilde}(X_T)}=1$, and $H^{\rk
    T}(M(\SA_{X_T}))^{N_{G}(X_T)}=2$.

  Consider first the case when $T=G(r,p,n-2)A_1$. By
  \cref{thm:4}\cref{it:2at},
  \[
    e_G\cdot B^{\SA,G}_{T} = e_G\cdot\{ \cx^{\SA,G}_{T}, \cx^{\SA,G}_{T,1}\}=
    \{e_G\cdot h_\varphi h_2\dotsm h_{n-2}h_n, \ e_G\cdot h_\varphi h_2'\dotsm
    h_{n-2}h_n\}
  \]
  is a basis of $\Ind_{X_T}^G$. Set $G_2 =G(r,2,n)$. Then $G_2 X_{T}= GX_{T}$,
  so $\Ind_{X_T}^{G_2} = \Ind_{X_T}^{G}$. Moreover, $G\subseteq G_2$ because
  $p$ is even, so $G_2$ acts on $\Ind_{X_T}^G$ and
  $\big(\Ind_{X_T}^G\big)^{G_2} \subseteq \big(\Ind_{X_T}^G\big)^{G}$. Since
  both spaces are two-dimensional, they must be equal. Thus every element in
  $\Ind_{X_T}^G$ is $G_2$-invariant. It follows from \cref{ssec:osp2} that
  \[
    \varphi \cdot h_\varphi= h_\varphi, \quad \varphi \cdot h_2= h_2',
    \quad\text{and}\quad \varphi \cdot h_j= h_j \text{ for $3\leq j\leq n$}.
  \]
  Thus,
  \[
    \varphi e_G \cdot \cx^{\SA,G}_{T} = e_G\cdot \cx^{\SA,G}_{T,
      1}\quad\text{and}\quad \varphi e_G\cdot \cx^{\SA,G}_{T,1} = \varphi^2
    e_G \cdot \cx^{\SA,G}_{T} = e_G \cdot \cx^{\SA,G}_{T},
  \]
  where the last equality holds because $\varphi^2\in G_2$. Because the
  subgroup of $\Gtilde$ generated by $\varphi$ maps surjectively onto
  $\Gtilde/G$, we see that $e_{G} \cdot\{ \cx^{\SA,G}_{T} +
  \cx^{\SA,G}_{T,1},\, \cx^{\SA,G}_{T} - \cx^{\SA,G}_{T,1} \}$ is a basis of
  $\Ind_{T}^G$ such that
  \begin{equation}
    \label{eq:5}
    e_{G} \cdot \cx^{\SA,G}_{T}+\ e_G\cdot \cx^{\SA,G}_{T,1} \in \big(\Ind_T^G
    \big)^{\Gtilde}\quad\text{and}\quad e_{G} \cdot \cx^{\SA,G}_{T} -\
    e_G\cdot \cx^{\SA,G}_{T,1} \in \big(\Ind_T^G\big)^{\epsilon_{\Gtilde}}.
  \end{equation}
    
  A similar argument applied to $T=G$ shows that $e_{G} \cdot \{
  \cx^{\SA,G}_G+ \cx^{\SA,G}_{G},\, \cx^{\SA,G}_G- \cx^{\SA,G}_{G} \}$ is a
  basis of $H^{n}(M(\SA))^G$ such that 
  \begin{equation}
    \label{eq:6}
    e_{G} \cdot \cx^{\SA,G}_{G}+\ e_G\cdot \cx^{\SA,G}_{G,1} \in
    \big(\Ind_G^G\big)^{\Gtilde} \quad\text{and}\quad e_{G} \cdot
    \cx^{\SA,G}_{G}-\ e_G\cdot \cx^{\SA,G}_{G,1} \in
    \big(\Ind_G^G\big)^{\epsilon_{\Gtilde}} .
  \end{equation}

  It follows from \cref{eq:5} and \cref{eq:6} that the character of the $\BBQ
  \Gtilde$-module $(\Ind_{X_T}^G)^G$ is equal to $1_{\Gtilde}+
  \epsilon_{\Gtilde}$ when $p$ and $n$ are even and $T\in \{G(r,p,n-2)A_1,
  G\}$.
\end{proof}

\section{Application: Semi-invariants and determinant-like
  characters}\label{sec:semi}


For our last application, we consider the situation when $G$ is a complex
reflection group, $\SA=\SA(G)$, and we extend the underlying scalar field for
cohomology, $G$-modules, and representations from $\BBQ$ to $\BBC$.  Using the
same ideas as in the arguments above, we give a short proof of the following
vanishing theorem due to Lehrer \cite[Thm.~1.3]{lehrer:vanishing}.

\begin{theorem}\label{thm:5}
  Suppose $G\subseteq \GL(V)$ is a complex reflection group, $\SA=\SA(G)$, and
  $L$ is a complex vector space that affords a representation of $G$ such that
  $L^{Z_G(H)}=0$ for every hyperplane $H\in \SA$. Then $\Hom_G
  \big(H^*(M(\SA)), L\big)=0$.
\end{theorem}

\begin{proof}
  Set $\CX=\CX(\SA,G)$. The proof is by induction on $\rk G$. We have seen in
  \cref{ssec:tdi} that if $\rk G =1$, then $G$ acts trivially on
  $H^*(M(\SA))$, and so the result holds. Suppose that $\rk G >1$ and that the
  result holds for complex reflection groups $G'$ with $\rk G'<\rk G$.
	
  By \cref{pro:bries} and Frobenius reciprocity we have
  \begin{equation}
    \label{eq:1}
    \begin{aligned}
      \Hom_G \big(H^*(M(\SA(G))), L\big)%
      &\cong \bigoplus_{X\in \CX} \Hom_G \big( \Ind_X^G, L \big)\\
      &\cong \bigoplus_{X\in \CX} \Hom_{N_G(X)} \big( H^{\cd X} ( M(\SA_{X})),
      L|_{N_G(X)} \big)\\
      &\cong \bigoplus_{X\in \CX} \Hom_{N_G(X)} \big(H^{\rk Z_G(X)} (
      M(\SA(Z_G(X)))), L|_{N_G(X)} \big).
    \end{aligned}
  \end{equation}
  The summand in \cref{eq:1} indexed by $X=V$ is $\Hom_{G} \big(H^{0}
  (M(\SA)), L \big)$, which is equal to zero by assumption. Suppose $X\in \CX$
  and $0<\cd X<\rk G$. Clearly $L^{Z_{Z_G(X)}(H)}=0$ for every hyperplane in
  $\SA(Z_G(X))$, and so by induction, $\Hom_{Z_G(X)} \big( H^{*}(
  M(\SA(Z_G(X)))), L|_{Z_G(X)} \big)=0$. Every $N_G(X)$-equivariant
  homomorphism from $H^{\rk Z_G(X)} ( M(\SA(Z_G(X))))$ to $L$ is
  $Z_G(X)$-equivariant, so
  \[
    \Hom_{N_G(X)} \big(H^{\rk Z_G(X)} ( M(\SA(Z_G(X)))), L|_{N_G(X)} \big) =0.
  \]
  Finally, the argument in \cref{ssec:cpp}\cref{eq:cn} shows that the summand
  in \cref{eq:1} indexed by $X=0$, namely $\Hom_{G} \big(H^{\rk G} ( M(\SA)),
  L \big)$, is also equal to zero. This completes the proof.
\end{proof}

A linear character $\xi$ of $G$ is called a \emph{determinant-like} character
if, for every reflection $s\in G$, $\xi(s)$ is a root of unity with order
equal to the order of $s$. Such a character satisfies the hypothesis of the
theorem. The determinant character, $\det \colon G\to \BBC^*$, is of course a
determinant-like character. If $G$ is a Coxeter group, then $\det$ is the sign
character, and is the only determinant-like character of $G$. For non-Coxeter
reflection groups, $\det\inverse$ is a determinant-like character that in
general is not equal to $\det$.

It follows from a result of Lehrer \cite{lehrer:rational} that the sign
character of a Weyl group, $W$, does not occur in $H^*(M(\SA(W)))$. The
following generalization to complex reflection groups is an immediate
consequence of \cref{thm:5}.

\begin{corollary}\label{cor:5}
  Suppose $G\subseteq \GL(V)$ is a non-trivial, complex reflection group and
  $\xi$ is a determinant-like character of $G$. Then $\xi$ does not occur in
  $H^*(M(\SA(G)))$.
\end{corollary}


\appendix
\section{Calculations in rank two}

In this appendix we collect the results for irreducible, rank two reflection
pairs. In the tables we use the notation for scalar matrices introduced in
\cref{ssec:ex}. Boldface indicates a coset representative, say $\bm{z}$, such
that $\SA(G\bm{z})= \emptyset$. For example $\bm{z_{12}}$ is the scalar matrix
with eigenvalue $\zeta_{12}= e^{2\pi i/12}$. In the row for $G=G_4$ in
\cref{tab:irp2}, $\langle G_{4}\bm{z_{12}} \rangle = G_7$ and hence
$\SA(\langle G_{4}\bm{z_{12}} \rangle )= \SG_7$, but $\SA(G_{4}\bm{z_{12}})=
\emptyset$.

\subsection{}
Suppose $G\sigma$ is a rank two reflection coset in \cref{ssec:rc}. The
arrangements $\SA(\langle G\sigma z\rangle)$ and groups $\langle
G\sigma\rangle$, where $z$ is a scalar transformation and $\SA(\langle G\sigma
z\rangle) \ne \SA(G)$, are given in \cref{tab:irp2}.

\begin{table}[htb]
  \renewcommand{\arraystretch}{1.2} 
  \caption{Arrangements $\SA(\langle G\sigma z\rangle)\ne \SA(G)$ and groups
    $\langle G\sigma \rangle$ in rank two}
  \begin{tabular} {>{$}r<{$} @{\hspace {2em}} >{$}c<{$} @{\hspace {2em}}
      >{$}c<{$} @{\hspace {2em}}>{$}l<{$} @{\hspace {3em}} l}
    \toprule \addlinespace
    G&\sigma z&\SA(\langle G\sigma z\rangle)&\langle G\sigma \rangle%
    & Notes \\ 
    \addlinespace \toprule \addlinespace
     &\varphi^p &\SA_2(r)&G_{r,p,2} &\\
     &\varphi^pz_{2r}^{-p} & \SA_2^0(2r)&G_{r,p,2}&$p$ odd\\ 
    \multirow{-3}*{$G_{r,r,2}$}%
     &\varphi^{p+1} z_{2r}\inverse%
              &\SA_2(2r)&G_{r, \gcd(r,p+1),2}&$r/p$ odd\\ 

    \addlinespace \midrule \addlinespace
     &\varphi^q &\SA_2(r)&G_{r,q,2}&\\
    \multirow{-2}*{$G_{r,p,2}$}&\varphi z_{2r}\inverse
              &\SA_2(2r)&G_{r,1,2}&\\ 
    \addlinespace \midrule \addlinespace
    G_{4,2,2}&\rho_3&\SG_6&G_6&\\
    \addlinespace \midrule \addlinespace

    G_{4}&z_3, z_4, \bm{z_{12}}&\SG_{5}, \SG_6, \SG_7%
                                            &G_4&\\
    G_{5}&z_4, \rho_4, \rho_2&\SG_{7}, \SG_{10}, \SG_{14}%
                                            &G_5,G_{10},G_{14}&\\ 
    G_{6}&z_3&\SG_{7}&G_6&\\
    G_{7}&\rho_4, \rho_2& \SG_{10}, \SG_{15}%
                                            &G_{10}, G_{15}&\\
    G_{8}&z_8, z_3, \bm{z_{24}}&\SG_{9}, \SG_{10}, \SG_{11}%
                                            &G_8&\\
    G_{9}&z_3&\SG_{11}&G_{11}&\\
    G_{10}&z_8&\SG_{11}&G_{11}&\\

    \addlinespace \midrule \addlinespace
     &z_8, \bm{z_{24}}, z_4,%
              &\SG_{9}, \SG_{11}, \SG_{13} &G_{12}& \\
    \multirow{-2}*{$G_{12}$}&z_3, \bm{z_{12}}%
              &\SG_{14}, \SG_{15} &G_{12}%
    &\multirow{-2}*{$\SG_{11}=\SG_{15}$}\\ 
    \addlinespace \midrule \addlinespace

    G_{13}&\bm{z_{8}}, \bm{z_{24}}, z_3%
              &\SG_{9}, \SG_{11},\SG_{15}&G_{13}%
    &$\SG_{13}=\SG_{9}$, $\SG_{11}=\SG_{15}$\\  
    G_{14}&z_8, z_4&\SG_{11}, \SG_{15}&G_{14}&$\SG_{11}=\SG_{15}$\\
    G_{15}&\bm{z_8}&\SG_{11}&G_{15}&$\SG_{11}=\SG_{15}$\\
    G_{16}&z_4, z_3, \bm{z_{12}}&\SG_{17}, \SG_{18}, \SG_{19}%
                                            &G_{16}&\\
    G_{17}&z_3&\SG_{19}&G_{17}&\\
    G_{18}&z_4&\SG_{19}&G_{18}&\\
    G_{20}&z_5, \bm{z_{20}}, z_4&\SG_{18}, \SG_{19}, \SG_{21}%
                                            &G_{20}&\\
    G_{21}&z_5&\SG_{19}&G_{21}\\
    G_{22}&z_5, \bm{z_{15}}, z_3&\SG_{17}, \SG_{19}, \SG_{21}&G_{22}&\\
    \addlinespace \bottomrule \addlinespace
  \end{tabular}
  \label{tab:irp2}
\end{table}

\subsection{}

Suppose $C$ is an irreducible reflection coset with rank equal to two, $G_0= C
C\inverse$, and $G_1 = \langle C \rangle$. Then for $i,j\in \{0,1\}$, either
$H^*( M(G_i) )^{G_j} = H^*( M( \SA(G) )^{G}$, where $G$ is irreducible, or
$H^*( M(G_i) )^{G_j} = H^*( M(\SA) )^{G}$, where $(\SA, G)$ is one of the
pairs in \cref{tab:rk2}. The table also records a reflection coset, $C$,
that gives rise to the reflection pair $(\SA, G)$.

\begin{table}[h]
  \renewcommand{\arraystretch}{1.2} \centering 
  \caption{Irreducible reflection pairs, rank two, not $(\SA(G), G)$}
  \begin{tabular} {>{$}r<{$} @{\hspace{2em}} >{$}c<{$} @{\hspace{2em}}
      >{$}l<{$} @{\hspace{3em}} l}
    \toprule \addlinespace
     {\SA} &  C& G&Notes \\
    \addlinespace \toprule \addlinespace%
    \SA_2^0(r)&&G_{r,p,2}&\\
    \SA_2(r)&\multirow{-2}*{$G_{r,r,2}\varphi^p$}
                     &G_{r,r,2}&\\
    \addlinespace \midrule \addlinespace
              &G_{r,r,2}\varphi z_{2r}\inverse&G_{r,r,2}&\\
    \multirow{-2}*{$\SA_2^0(2r)$}%
              &G_{r,r,2}\varphi^pz_{2r}^{-p}&G_{r,p,2}&$p$ odd\\ 
    \addlinespace \midrule \addlinespace
              & G_{r,r,2}\varphi^{p+1}z_{2r}\inverse&G_{r,r,2}&$r/p$ odd\\
    \multirow{-2}*{$\SA_2(2r)$}%
              &G_{r,p,2}\varphi z_{2r}\inverse &G_{r,p,2}&\\

    \addlinespace \midrule \addlinespace
    \SA_2(4)& G_{4,2,2}\rho_3& G_6\\
    \SG_{5}&G_4z_3, G_{5}\rho_2, G_{5}\rho_4 &G_4, G_{14}, G_{10}\\
    \SG_{6}&G_{4,2,2}\rho_3, G_4z_4& G_{4,2,2}, G_4&\\
    \SG_{7} &G_{4}\bm{z_{12}}, G_{5} z_4, G_{6} z_3, G_{7}\rho_2,
              G_{7}\rho_4  &G_4, G_5, G_6, G_{15}, G_{10}\\
    \SG_{9} &G_{8} z_8, G_{12} z_8, G_{13}\bm{z_8}& G_8, G_{12}, G_{13}\\
    \SG_{10} &G_{5}\rho_4, G_{7}\rho_4, G_{8} z_3&G_5, G_7, G_8\\
    
    \addlinespace \midrule \addlinespace
    {}&G_8\bm{z_{24}}, G_9z_3, G_{10}z_8,G_{12}\bm{z_{24}}
                     &G_8, G_9, G_{10}, G_{12}\\
    \multirow{-2}*{$\SG_{11}$} &G_{13}\bm{z_{24}}, G_{14} z_8,
                                 G_{15}\bm{z_{8}}& G_{13},G_{14},G_{15}\\
    \addlinespace \midrule \addlinespace
    \SG_{13} &G_{12} z_4&G_{12}\\
    \SG_{14} &G_{5}\rho_2, G_{15} z_3&G_5, G_{12}\\
    \SG_{15}&G_{7}\rho_2, G_{12}\bm{z_{12}}, G_{13} z_3, G_{14} z_4
                     &G_7, G_{12}, G_{13}, G_{14}\\
    \SG_{17} &G_{16} z_4, G_{22} z_5& G_{16}, G_{22} \\
    \SG_{18} &G_{16} z_3, G_{20} z_5 & G_{16}, G_{20}\\
    \addlinespace \midrule \addlinespace
    {}& G_{16}\bm{z_{12}}, G_{17} z_3, G_{18} z_4&G_{16},G_{17},G_{18}\\
    \multirow{-2}*{$\SG_{19}$}%
              &G_{20} \bm{z_{20}}, G_{21} z_5,
                G_{22}\bm{z_{15}}&G_{20}, G_{21}, G_{22}\\
    \addlinespace \midrule \addlinespace
    \SG_{21} &G_{20} z_4, G_{22} z_3, G_{19} z_5& G_{20}, G_{22}, G_{19}\\

    \addlinespace \bottomrule \addlinespace
  \end{tabular}
  \label{tab:rk2}
\end{table}

\subsection*{Acknowledgments}
The research of this work was supported by the DFG (Grant \#RO 1072/19-1 to
G.~R\"ohrle). J.M.~Douglass would like to acknowledge that some of this
material is based upon work supported by, and while serving at, the National
Science Foundation. Any opinion, findings, and conclusions or recommendations
expressed in this material are those of the authors and do not necessarily
reflect the views of the National Science Foundation.

We are grateful to the referee for valuable suggestions that led to
significant improvements in the paper.


\bibliographystyle{plain}


\end{document}